%% file: CENTRAL_FILE.tex
\documentclass[12pt]{article}

\usepackage{amssymb}
\usepackage{amsfonts}
\usepackage[latin1]{inputenc}
\usepackage{enumerate}
\usepackage{tkz-graph}
\usetikzlibrary{arrows}
\usetikzlibrary{positioning}

\usepackage[bottom]{footmisc}

\usepackage[english]{babel}
\usepackage{graphicx}
\usepackage{amsmath,amsthm}
\usepackage{breqn}
\usepackage{mathrsfs}
\usepackage{psfrag}
\usepackage{color}
\usepackage{indentfirst}
\usepackage{subfigure, epsfig, epstopdf}
\usepackage{hyperref}

\usepackage{upgreek}
\usepackage{bm}

\usepackage{calligra}

\usepackage{tikz-cd}

\newtheorem {lemma}[equation]{Lemma}
\newtheorem {corollary}[equation]{Corollary}
\newtheorem {theorem}[equation]{Theorem}
\newtheorem {proposition}[equation]{Proposition}

\theoremstyle {definition}
\newtheorem {definition}[equation]{Definition}

\theoremstyle{remark}
\newtheorem{remark}[equation]{Remark}

\numberwithin{equation}{subsection}

\input{macros.tex}

\begin{document}

\title{The Ideal Structure of Steinberg Algebras}
\author{Paulinho Demeneghi}
\date{}

\maketitle

\abstract{Given an ample action of an inverse semigroup on a locally compact
	and Hausdorff topological space, we study the ideal structure
	of the crossed product algebra associated with it. By
	developing a theory of
	induced ideals, we manage to prove that every ideal in the crossed
	product algebra may be obtained as the intersection of ideals induced
	from isotropy group algebras. This can be interpreted as an algebraic
	version of the Effros-Hahn conjecture. Finally, as an application of
	our result, we study the ideal structure of a Steinberg algebra associated
	with an ample groupoid by interpreting it as an inverse semigroup crossed
	product algebra.}

\section{Introduction}

There is a celebrated conjecture which has motivated most of the works
in the study of ideals in crossed product C*-algebras since about fifty
years ago, namely the Effros-Hahn conjecture \cite{effros1967}.
The original conjecture states that every primitive ideal
in the crossed product of a commutative C*-algebra by a locally compact
group should be induced from a primitive ideal in the C*-algebra of some
isotropy group.

It was proved by Sauvageout in \cite{Sauvageot} for the case of discrete
amenable groups and, since then, it has been extended to various others
contexts. Gootman and Rosenberg \cite{GootmanRosenberg} have proved
a version for locally compact groups acting on non-commutative C*-algebras.
Renault has also introduced a version of the Effros-Hahn conjecture in
\cite{Renault1987} for groupoid C*-algebras, an entirely different setting.
And Renault's results were later refined by Ionescu and Williams in
\cite{IonescuWilliams}.

We should also mention Dokuchaev and Exel work in \cite{DokuchaevExel}
and Steinberg work in \cite{Steinberg} and \cite{Steinberg2}.
Dokuchaev and Exel have
introduced the conjecture in an algebraic fashion,
the algebraic partial crossed
product $\Lc{X} \rtimes G$, where $G$ is a discrete group partially
acting on a locally compact and totally disconnected topological space
$X$, and $\Lc{X}$ is the algebra consisting of all locally constant,
compactly supported functions on $X$, taking values in a given field $K$.
Steinberg introduced a notion of an algebra associated with an
ample groupoid $\G$ over a given field $K$, known as Steinberg algebras
nowadays. He obtained a remarkable number of results
for these algebras and,
among them, a theory of induction of modules from isotropy groups.

It is well known that the main object of study in \cite{DokuchaevExel},
the algebra $\Lc{X} \rtimes G$, may also be described as the Steinberg
algebra \cite{Steinberg} for the transformation groupoid associated with
the partial action of $G$ on $X$. Hence, Steinberg results may be
applied to $\Lc{X} \rtimes G$ as well. 

The question that arises in this moment is: could Dokuchaev and Exel
results be generalized for Steinberg algebras? The answer is affirmative
and, in this paper we focus in showing this.

For that task, as our main object of interest, we first concentrate
on crossed product algebras of the form
$\Lcs$, where $S$ is an inverse semigroup.
In fact, in the last part of this
paper, we show that every Steinberg algebra associated with an ample
groupoid (not necessarily Hausdorff) over a given field $K$ can be
realized as an inverse semigroup crossed product of the form $\Lcs$.
Similar results may be found in \cite{Beuter} and
\cite{Hazrat2018}.

This paper is structured in four main parts. First, we introduce
an algebraic notion of a Fell Bundle over an inverse semigroup, inspired by
\cite{Exel2011}, and then build the cross-sectional algebra associated with it.
Next, we show that an ample action
of an inverse semigroup $S$ over a locally compact and Hausdorff
topological space $X$ induces a Fell Bundle $\mathcal{B}^{\theta}$,
referred
as the semi-direct product bundle, and define the crossed product
algebra
$\Lc{X} \rtimes S$ as the cross-sectional algebra of the semi-direct
product bundle.
We then develop a theory linking representations of the crossed product
algebra and
covariant representations of the ample system $(\theta,S,X)$,
obtaining results of integration and disintegration.

The second part is dedicated to present the theory of induction
of ideals from isotropy groups algebras and to present the basics
of the induction process.
Based on Dokuchaev and Exel work, we study the relationship between the
\emph{input} ideal in the isotropy group algebra and its corresponding
\emph{output} induced ideal. It turns out that, for a point $x$ in $X$,
when inducing from the isotropy group $G_x$, not all
ideals in $KG_x$ play a relevant role. Those which do, we call
\emph{admissible}, inspired in Dokuchaev and Exel terminology.
We show in (\ref{AdmissibleInjectivity}) that, for every ideal
$I \trianglelefteq KG_x$, there exists a unique admissible ideal
$I' \subseteq I$, which induces the same ideal of $\Lcs$
as $I$ does. Thus, the correspondence $I \mapsto \Ind{x}{I}$
is seen to be an one-to-one mapping from the set of admissible ideals
in $KG_x$ to the set of ideals in $\Lcs$.

In the third part, we generalize Dokuchaev and Exel version of the
Effros-Hahn conjecture (\cite[Theorem 6.3]{DokuchaevExel}) 
for $\Lcs$, namely Theorem (\ref{MainResult}) which states that
every ideal of $\Lcs$ is given as the intersection of ideals induced from
isotropy groups. The method of the proof is inspired in \cite{DokuchaevExel}
and does not rely on measure theoretical or analytical tools. The strategy
adopted is as follows: given an ideal $J$ of $\Lcs$, we first choose a
representation $\pi$ of $\Lcs$ whose null space coincides with $J$. Through
the theory of integration and disintegration constructed before, we
then build another representation, which we call the \emph{discretization}
of $\pi$, as done in \cite{DokuchaevExel}, whose null space coincides
with that of $\pi$, and hence also with $J$. The discretized
representation is seen to decompose as a direct sum of sub-representations,
which are finally shown to be equivalent to an induced representation,
and hence the initially given ideal $J$ is seen to coincide with the
intersection of the null spaces of the various induced representations
involved, each of which is then an induced ideal.

Finally, in the last part of this paper, as a consequence of
Theorem (\ref{MainResult})
we show that every Steinberg algebra over a given field $K$ is isomorphic
to an inverse semigroup crossed product of the form $\Lcs$ and
the induction theory introduced by Steinberg in
\cite{Steinberg} and \cite{Steinberg2} is compatible
with our theory through the given isomorphism. So, our results can
be all applied to Steinberg algebras.

This paper has resulted from the author's work as a PhD student in Florian\'opolis
under the advice of professor Ruy Exel.

\markboth{}{Fell Bundle}
\input{Fell_Bundle.tex}

\markboth{}{Induction process}
\input{Induction.tex}

\markboth{}{Representation}
\input{representation.tex}

\markboth{}{Steinberg}
\input{Steinberg.tex}

\thispagestyle{plain}
\addcontentsline{toc}{chapter}{References}
\markboth{}{References}

\bibliographystyle{abnt}
\bibliography{CENTRAL_FILE}

\end{document}

%% file: macros.tex
\newcommand{\iprod}[2]{\langle#1 , #2\rangle}

\DeclareMathAlphabet{\mathcalligra}{T1}{calligra}{m}{n}
\newcommand{\algb}{\kern-2pt\mathcalligra{Alg}\kern1pt(\mathcal{B})}

\DeclareMathOperator{\Span}{span}
\DeclareMathOperator{\supp}{supp}
\DeclareMathOperator{\orb}{Orb}
\DeclareMathOperator{\orbG}{\mathbf{Orb}}
\DeclareMathOperator{\ind}{Ind}
\DeclareMathOperator{\Mod}{mod}

\newcommand\bool[1]{[{\scriptstyle #1}]\,} 
\newcommand\Lc[1]{\mathscr{L}_c({#1})} 
\newcommand\ran[1]{X_{{#1}{#1}^{\ast}}} 
\newcommand\dom[1]{X_{{#1}^{\ast}{#1}}} 
\newcommand\germ[2]{\left[ {#1} , {#2} \right]} 

\newcommand\germo[1]{\left[ {#1} , x \right]} 
\newcommand{\G}{\mathcal{G}} 
\newcommand{\SG}{\mathcal{S}(\mathcal{G})} 
\newcommand\Set[2]{\left\{ {#1} \text{ : } {#2} \right\}} 
\newcommand\A[1]{A_{K}({#1})} 
\newcommand\Ind[2]{\ind_{#1}(#2)}
\newcommand{\balpha}{\bar{\alpha}}
\newcommand{\discr}{\Pi \times U} 
\newcommand{\till}{\tilde{L}}
\newcommand{\tilh}{\tilde{G}}
\newcommand{\mL}{\mathbf{L}}
\newcommand{\mH}{\mathbf{G}}
\newcommand{\mM}{\mathbf{M}}
\newcommand{\mGamma}{\mathbf{\Gamma}}

\newcommand{\Lcs}{\Lc{X}\rtimes S}
\newcommand{\bmdelta}{\bm{\updelta}}

\newcommand{\Triple}{
\mathcal{B}^{\alpha} = \left(
\left\{ B_s \right\}_{s \in S},
\left\{ \mu_{s,t} \right\}_{s,t \in S},
\left\{ j_{t,s} \right\}_{s,t \in S, s \leq t}
\right)
}

\newcommand\restr[2]{{
		\left.\kern-\nulldelimiterspace 
		#1 
		\vphantom{\big|} 
		\right|_{#2} 
}}

%% file: Fell_Bundle.tex
\section{Inverse semigroup crossed products}

In this section we explore inverse semigroup crossed
product algebras and its
universal property. Actually, this algebras have already came to
light meanwhile
this work was in progress in
\cite{Beuter} and \cite{Hazrat2018}, for example.
However, we prefer to build
it in a slightly different fashion.

We introduce an intermediary step, namely, an algebraic notion of a Fell
bundle over an inverse semigroup, based on Exel's paper \cite{Exel2011}.
With this in hand we then build the cross-sectional algebra associated
with a Fell bundle and discuss an universal property with relation to
representations. Finally, from an action of an inverse semigroup, when
possible, we construct a Fell Bundle, named the semi-direct product bundle
associated with the action. Turns out that the cross-sectional algebra of a
semi-direct product bundle ``coincides" with the crossed product algebra
in the sense of \cite{Beuter} and \cite{Hazrat2018}.

From our context, we see that the crossed product algebra arising from
actions of inverse semigroups on locally compact, totally disconnected
and Hausdorff spaces inherits an universal property with relation to
representations.

\subsection{Fell Bundles over Inverse Semigroups}

We assume that the reader is familiar with the notion of an inverse
semigroup and its basics notations: the semigroup is denoted by $S$,
the involutive anti-homomorphism by $\ast$, and the set of all
idempotent elements by $E(S)$.

\textbf{Throughout this paper we fix a field $K$.}

Most results in this section are still valid in the more general case
obtained by replacing $K$ by a commutative ring with identity.
However, we prefer to maintain $K$ as a field all long this paper since
for the main results, this assumption is needed.

\begin{definition} \label{FellBundle}
	A \emph{Fell Bundle} over an inverse semigroup $S$ is a triple
	\begin{equation*}
	\mathcal{B} = \left(
	\left\{ B_s \right\}_{s \in S}, 
	\left\{ \mu_{s,t} \right\}_{s,t \in S},
	\left\{ j_{t,s} \right\}_{s,t \in S, s \leq t}
	\right)
	\end{equation*}
	such that, for each $s,t \in S$
	\begin{enumerate}[(a)]
		\item $B_s$ is a $K$-vector space;
		\item $\mu_{s,t}:B_s \otimes B_t \to B_{st}$ is a $K$-linear map;
		\item $j_{t,s}: B_s \to B_t$ is a $K$-linear injective map for every
		$s \leq t$.
	\end{enumerate}
	It is moreover required that for every $r,s,t\in S$,
	\begin{enumerate}[(i)]
		\item \label{Associativity} $\mu_{rs, t}\left(\mu_{r, s}(a \otimes b) \otimes c \right) =
		\mu_{r,st} \left( a \otimes \mu_{s,t}(b \otimes c) \right)$ for every
		$a \in B_r$, $b \in B_s$ and $c \in B_t$;
		\item \label{NoPatology}
		$\Span \left\{ \mu_{ss^*, s} \left(\mu_{s,s^*}(B_s \otimes B_{s^*}) \otimes B_s \right)  \right\}
		= B_s$;
		\item \label{Funtorial} $j_{t,r}= j_{t,s} \circ j_{s,r}$ if $r\leq s\leq t$;
		\item \label{MultIndepend} if $r\leq r'$ and $s \leq s'$, then the diagram
		\[ \begin{tikzcd}
		B_r \otimes B_s \arrow{r}{\mu_{r, s}} \arrow[swap]{d}{j_{r',r} \otimes j_{s',s}} & B_{rs} \arrow{d}{j_{r's',rs}} \\
		B_{r'} \otimes B_{s'} \arrow[swap]{r}{\mu_{r',s'}} & B_{r's'}
		\end{tikzcd}
		\]
		commutes.
	\end{enumerate}
\end{definition}

If $s \leq t$, we shall use the map $j_{t,s}$ to identify $B_s$ as a subspace of $B_t$.
The last axiom then says that the
multiplication operation is compatible with such an identification.

There are some immediate consequences of the definition.

\begin{proposition}\hspace{1cm}
	\begin{enumerate}[(a)]
		\item If $e \in E(S)$, then $B_e$ is an associative $K$-algebra.
		\item For every $s \in S$, the map $j_{s,s}$ is the identity map on $B_s$.
		\item If $e,f \in E(S)$ and $e \leq f$, then $j_{f,e}(B_e)$ is a two-sided ideal in $B_f$.
	\end{enumerate}
\end{proposition}

\begin{proof}
The first item is obvious. For the second item, let $s \in S$ and notice that
$j_{s,s}$ is an injective linear map from $B_s$ to itself, which is idempotent
by (\ref{FellBundle}.\ref{Funtorial}). Therefore, $j_{s,s}$ must be the identity map on $B_s$, as stated.
Finally, with respect to (c), let $a \in B_e$, $b \in B_f$ and notice that
\begin{equation*}
j_{f,e}(a) \cdot b
= \mu_{f,f} \Big( j_{f,e}(a) \otimes j_{f,f}(b) \Big)
\overset{(\ref{FellBundle}.\ref{MultIndepend})}{=} j_{ff,ef}
	\Big(\mu_{e,f}(a \otimes b) \Big)
= j_{f,e} \Big( \mu_{e,f}(a \otimes b) \Big)
\in j_{f,e}(B_e),
\end{equation*}
and similarly $b \cdot j_{f,e}(a) \in j_{f,e}(B_e)$. This shows that $j_{f,e}(B_e)$ is
a two-sided ideal in $B_f$, as desired.
\end{proof}

\begin{definition}
	A \emph{pre-representation} of a Fell bundle $\mathcal{B}=\left\{ B_s \right\}_{s \in S}$
	in an algebra $A$
	is a family
	$\Pi = \left\{ \pi_s \right\}_{s \in S}$ of linear maps
	\begin{equation*}
	\pi_s: B_s \to A
	\end{equation*}
	such that, for all $s,t \in S$ and all $a \in B_s$ and $b \in B_t$, we have
	\begin{enumerate}[(i)]
		\item $\pi_{st}\left(\mu_{s,t}(a \otimes b)\right) = \pi_s(a) \pi_t(b)$.
		
		Furthermore, $\Pi$ is a representation if it satisfies
		\item $\pi_t \circ j_{t,s} = \pi_s$, whenever $s \leq t$.
	\end{enumerate}
	
	In this context, if $V$ is a $K$-vector space and $A=L(V)$, then we shall
	say that $\Pi$ is a representation of $\mathcal{B}$ on $V$.
\end{definition}

\begin{definition}
	The \emph{cross-sectional} algebra of $\mathcal{B}$, denoted by $\algb$,
	is the universal algebra generated by the disjoint union
	\begin{equation*}
	\dot{\bigcup_{s \in S}} B_s,
	\end{equation*}
	subject to the relations stating that the natural maps
	\begin{equation*}
	\pi_s^u: B_s \to \algb
	\end{equation*}
	form a representation of $\mathcal{B}$ in $\algb$.
\end{definition}

The existence of $\algb$ is clear, as its uniqueness, up to isomorphism. For
convenience, we spell out its universal property.

\begin{proposition}\label{UniversalProp}
	The cross-sectional algebra $\algb$ is an algebra and
	$\Pi^u= \left\{ \pi_s^u \right\}_{s \in S}$
	is a representation of $\mathcal{B}$ in $\algb$. Furthermore, given any
	representation $\Pi= \left\{ \pi_s \right\}_{s \in S}$ of the Fell Bundle
	$\mathcal{B}$ in an algebra $A$, there exists a unique homomorphism
	$\Phi: \algb \to A$ such that $\Phi \circ \pi_s^u = \pi_s$ for all $s \in S$.
\end{proposition}

It will be useful to have a more concrete description of $\algb$ as follows.
Let
\begin{equation*}
\mathcal{L}(\mathcal{B})=\bigoplus_{s \in S} B_s.
\end{equation*}
For each $s \in S$ and $b_s \in B_s$, we denote by $b \bmdelta_s$
the element of
$\mathcal{L}(\mathcal{B})$ whose coordinates are equal to zero, except for the
coordinate corresponding to $s$, which is equal to $b$.
Then, it is clear that any element $b \in \mathcal{L}(\mathcal{B})$ can be represented uniquely in the form
\footnote{All sums considered in this paper are finite. Either
because the summands are indexed on a finite set, or all but a
finitely many summands are zero.}
\begin{equation*}
b = \sum_{s \in S} b_s \bmdelta_s.
\end{equation*}
Define a multiplication on $\mathcal{L}(\mathcal{B})$ such that
\begin{equation*}
(b_s \bmdelta_s)(b_t \bmdelta_t)=\mu_{s,t}(b_s \otimes b_t) \bmdelta_{st}
\end{equation*}
for all $s,t \in S$, $b_s \in B_s$ and $b_t \in B_t$.

Then, with (\ref{FellBundle}.\ref{Associativity}), we can prove that
$\mathcal{L}(\mathcal{B})$ is an associative
$K$-algebra.

\begin{definition}\label{PiZero}
	Let $\Pi^{0} = \left\{ \pi_s^{0} \right\}_{s \in S}$ be
	the collection of maps such that,
	for each $s \in S$, $\pi_s^{0}: B_s \to \mathcal{L}(\mathcal{B})$ is given by
	\begin{equation*}
	\pi_s^{0}(b_s) = b_s \bmdelta_s.
	\end{equation*}
\end{definition}

In this fashion, $\Pi_{0}$ is a pre-representation of $\mathcal{B}$ in $\mathcal{L}(\mathcal{B})$ which
is universal in the following sense.

\begin{proposition}\label{correspondence}
	Let $A$ be an algebra. If $\Pi=\left\{ \pi_s \right\}_{s \in S}$ is a pre-representation of $\mathcal{B}$
	in $A$, then the map $\Phi: \mathcal{L}(\mathcal{B}) \to A$, given by
	\begin{equation*}
	\Phi \left( \sum_{s \in S} b_s \bmdelta_s \right) = \sum_{s \in S} \pi_s(b_s)
	\end{equation*}
	is a homomorphism. Conversely, given any homomorphism $\Phi: \mathcal{L}(\mathcal{B}) \to A$, consider
	for each $s \in S$, the map $\pi_s: B_s \to A$ given by
	\begin{equation*}
	\pi_s=\Phi \circ \pi_s^{0}.
	\end{equation*}
	Then, $\Pi = \left\{ \pi_s \right\}_{s \in S}$ is a pre-representation of $\mathcal{B}$ in $A$.
	Furthermore, the correspondences $\Pi \mapsto \Phi$ and $\Phi \mapsto \Pi$ are each other inverses,
	giving bijections between the set of all homomorphisms from $\mathcal{L}(\mathcal{B}) \to A$ and the
	set of all pre-representations of $\mathcal{B}$ in $A$.
\end{proposition}

\begin{proposition}\label{RedundanceIdeal}
	Let $\mathcal{N}$ be the linear subspace of $\mathcal{L}(\mathcal{B})$ spanned by the set
	\begin{equation*}
	\left\{ b_s \bmdelta_s - j_{t,s}(b_s) \bmdelta_t : s,t \in S, s \leq t, b_s \in B_s \right\}.
	\end{equation*}
	Then, $\mathcal{N}$ is a two-sided ideal of $\mathcal{L}(\mathcal{B})$.
\end{proposition}

\begin{proof}
Given $r,s,t \in S$ such that $s \leq t$, let $b_s \in B_s$ and $b_r \in B_r$. Notice that, by
(\ref{FellBundle}.\ref{MultIndepend}), we have
$\mu_{t,r} \circ \left( j_{t,s} \otimes j_{r,r} \right)
= j_{tr,sr} \circ \mu_{s,r}$ and so
\begin{align*}
\left( b_s \bmdelta_s - j_{t,s}(b_s) \bmdelta_t \right) b_r \bmdelta_r
& = \mu_{s,r}(b_s \otimes b_r) \bmdelta_{sr} - \mu_{t,r}\left( j_{t,s}(b_s) \otimes b_r \right) \bmdelta_{tr} \\
& =\mu_{s,r} (b_s \otimes b_r) \bmdelta_{sr} - j_{tr,sr} \left( \mu_{s,r}(b_s \otimes b_r) \right) \bmdelta_{tr}
\in \mathcal{N}.
\end{align*}
Therefore, we conclude that $\mathcal{N}$ is a right ideal and, similarly, we can show that
$\mathcal{N}$ is a left ideal.
\end{proof}

Notice that, in the context of Proposition (\ref{correspondence}),
$\Phi$ vanishes on $\mathcal{N}$ if and only if $\pi_t \circ j_{t,s} = \pi_s$, whenever $s \leq t$. Then, we
immediately have the following proposition.

\begin{proposition}\label{RepvsVanish}
	In the context of the correspondence $\Phi \leftrightarrow \Pi$ of
	(\ref{correspondence}), $\Phi$ vanishes on $\mathcal{N}$ if, and
	only if,
	$\Pi$ is a representation.
\end{proposition}

We now establish a very important representation of $\mathcal{B}$.

\begin{corollary}
	For each $s \in S$, let $\pi_s^{+} = q \circ \pi_s^{0}$,
	where $\pi_s^0$ is like in
	(\ref{PiZero}) and
	$q: \mathcal{L}(\mathcal{B}) \to \mathcal{L}(\mathcal{B})/ \mathcal{N}$
	is the quotient map. Then,
	$\Pi^{+} = \left\{ \pi^{+}_s \right\}_{s \in S}$
	is a representation of $\mathcal{B}$ in $\mathcal{L}(\mathcal{B})/ \mathcal{N}$.
\end{corollary}

\begin{remark}\label{remarkDelta}
	We shall denote by $b \Delta_s$ the image of $b \bmdelta_s$ in
	$\mathcal{L}(\mathcal{B})/ \mathcal{N}$ by the quotient map
	$q: \mathcal{L}(\mathcal{B}) \to \mathcal{L}(\mathcal{B})/ \mathcal{N}$.
\end{remark}

The importance of the representation $\Pi^{+}$ resides in the following result.

\begin{proposition}\label{concretedescription}
	The algebra  $\mathcal{L}(\mathcal{B})/ \mathcal{N}$ possesses the universal property described
	in (\ref{UniversalProp}) with respect to the representation $\Pi^{+}$.
\end{proposition}

\begin{proof}
Let $\Pi = \left\{ \pi_s \right\}_{s \in S}$ be any representation of $\mathcal{B}$ in an algebra $A$
and $\Psi: \mathcal{L}(\mathcal{B}) \to A$ be given as in (\ref{correspondence}) in terms of $\Pi$.
By (\ref{RepvsVanish}), $\Psi$ vanishes at
$\mathcal{N}$ and hence it factors through $\mathcal{L}(\mathcal{B})/ \mathcal{N}$ giving a
homomorphism $\Phi: \mathcal{L}(\mathcal{B})/ \mathcal{N} \to A$ such that
\begin{equation*}
\Phi(b_s \Delta_s) = \Phi(q(b_s \bmdelta_s))=\pi_s(b_s)
\end{equation*}
whenever $b_s \in B_s$. Furthermore, notice that
\begin{equation*}
\pi_s (b_s) = \Phi(q(b_s \bmdelta_s )) = \Phi(q(\pi^{0}_s(b_s)))=\Phi(\pi^{+}_s(b_s))
\end{equation*}
for every $s \in S$, as desired. It is also clear that such
$\Phi$ must be unique.
\end{proof}

We then have an immediate corollary.

\begin{corollary}
	There exists an isomorphism  $\Theta: \mathcal{L}(\mathcal{B})/ \mathcal{N} \to \algb$,
	such that $\Theta \circ \pi^{+}_s = \pi^{u}_s$, for every $s \in S$.
\end{corollary}

We shall henceforth identify $\mathcal{L}(\mathcal{B})/ \mathcal{N}$ and $\algb$, keeping in
mind that this identification caries $\pi^{+}_s$ to $\pi^{u}_s$, for every $s \in S$.

Before we end this section, we introduce some important ingredients.

\begin{definition}
	A representation $\Pi = \left\{ \pi_s \right\}_{s \in S}$ of a Fell bundle $\mathcal{B}$
	on a $K$-vector space $V$ is \emph{non-degenerate} if
	$$\Span \Set{\pi_s(b)\xi}{s \in S, b \in B_s, \xi \in V} = V.$$
\end{definition}

Notice that, if $b \in B_{ss^*}$ and $c \in B_s$, then setting
$a=\mu_{ss^{\ast},s}(b \otimes c)$ we have
$$\pi_s(a)=\pi_{ss^\ast}(b)\pi_s(c).$$
Hence, by (\ref{FellBundle}.\ref{NoPatology}), a representation of $\mathcal{B}$ on $V$ is
non-degenerate if and only if
$$\Span \Set{\pi_e(b)\xi}{e \in E(S), b \in B_e, \xi \in V} = V.$$

\begin{proposition}\label{degVSdeg}
	In the context of Proposition (\ref{UniversalProp}), let $A=L(V)$ for some vector space
	$V$. Then, $\Pi$ is non-degenerate if and only if $\Phi$ is non-degenerate.
\end{proposition}

\begin{proof}
Suppose $\Pi$ is non-degenerate and let $\xi \in V$ be such that $\xi=\pi_s(b_s)\eta$
for some $b_s \in B_s$ and $\eta \in V$.
Then
$$ \xi = \pi_s(b_s) = \Phi(\pi_s^{u}(b_s)) \eta.$$
Since the vectors $\xi$ of the above form spans $V$, $\Phi$ is non-degenerate.
Conversely, suppose $\Phi$ is non-degenerate and let $\xi \in V$ be such that
$\xi = \Phi(b)\eta$ where $b=q(\sum_{s \in S}b_s\bmdelta_s) \in \algb$ and $\eta \in V$.
Then
$$ \sum_{s \in S} \pi_s(b_s)\eta = \sum_{s \in S} \Phi (\pi_s^u(b_s)) \eta = \Phi(b)\eta = \xi.$$
Since the vectors $\xi$ of the above form spans $V$, $\Pi$ is non-degenerate.
\end{proof}

\subsection{Inverse semigroup actions and algebraic crossed products}

The aim of this subsection is to construct a Fell bundle from an action
of an inverse semigroup on an algebra. Unfortunately, this is not always
possible, the problem in the construction will appear in the axioms
(\ref{Associativity}) and (\ref{NoPatology}) of Definition (\ref{FellBundle}),
as we shall see.

Let $X$ be any set, we denote by $\mathcal{I}(X)$ the inverse semigroup
formed by all bijections between subsets of $X$, under the operation
given by composition of functions in the largest domain in which the
composition may be defined. We now present the definition of an action
of an inverse semigroup on an algebra.

\begin{definition}\label{algaction}
	Let $S$ be an inverse semigroup and let $A$ be an algebra. An \emph{action} of $S$ on $A$
	is a semigroup homomorphism
	\begin{equation*}
	\alpha: S \to \mathcal{I}(A)
	\end{equation*}
	such that
	\begin{enumerate}[(i)]
		\item for every $s \in S$, the domain (and hence also the range) of $\alpha_s$ is a two sided
		ideal of $A$ and $\alpha_s$ is a homomorphism;
		\item the linear span of the union of the domains of all the $\alpha_s$ coincides with $A$.
	\end{enumerate}
	The triple $(\alpha,S,A)$ is called an (algebraic) \emph{dynamical system}.
\end{definition}

For every $e \in E(S)$, we denote by $A_e$ the domain of $\alpha_e$. Therefore, for each $s \in S$,
we have that $\alpha_s$ is a homomorphism from $A_{s^{*}s}$ to $A_{ss^{*}}$.

\textbf{Throughout this subsection we fix an algebraic
dynamical system $(\alpha, S, A)$,
in order to describe the construction of the Fell bundle.}

We begin the construction defining, for each $s \in S$, the
``fiber" $B_s = \left\{ (a,s) \in A \times S \text{ : } a \in A_{ss^{*}} \right\}$.
To avoid excessive use of parentheses, 
we shall write $a \bmdelta_s$ to refer to $(a,s)$ whenever $a \in A_{ss^{*}}$.

The linear structure of $B_s$ is borrowed from $A_{ss^{*}}$, while the
multiplication operation is defined on elementary tensors by
\begin{equation*}
\begin{array}{cccc}
\mu_{s,t}: & B_s \otimes B_t & \to & B_{st} \\
& a \bmdelta_s \otimes b \bmdelta_t & \mapsto
& \alpha_s(\alpha_{s^*}(a)b) \bmdelta_{st}.
\end{array}
\end{equation*}
We then define the inclusion maps naturally
\begin{equation*}
\begin{array}{cccc}
j_{t,s}: & B_s & \to & B_t \\
& a \bmdelta_s & \mapsto & a \bmdelta_t
\end{array}
\end{equation*}
whenever $s,t \in S$ with $s \leq t$, which finally leads to a triple
\begin{equation}\label{Triple}
\Triple.
\end{equation}
In order to the triple $\Triple$ to be a Fell Bundle over $S$, we must
worry about axioms (\ref{FellBundle}.\ref{Associativity}-\ref{MultIndepend}).
Axioms (\ref{Funtorial}) and (\ref{MultIndepend}) are easy to see, but
as previous commented, axioms (\ref{Associativity}) and (\ref{NoPatology})
may not hold. To identify the origin of the problem with axiom
(\ref{Associativity}), let $a \bmdelta_r \in B_r$, $b \bmdelta_s \in B_s$
and $c \bmdelta_t \in B_t$, for $r,s,t\in S$, and notice that, computing
initially the left hand side of (\ref{FellBundle}.\ref{Associativity}),
we obtain:
\begin{align}\label{lefthandside}
\mu_{rs, t}\bigg(\mu_{r, s}(a \bmdelta_r \otimes b \bmdelta_s) \otimes c \bmdelta_t \bigg)
& = \mu_{rs, t}\bigg( \alpha_r \Big(\alpha_{r^{\ast}}(a)b \Big) \bmdelta_{rs} \otimes c \bmdelta_t \bigg) \nonumber \\
& = \alpha_{rs} \bigg( \alpha_{s^{\ast}r^{\ast}}
\Big(\alpha_r(\alpha_{r^{\ast}}(a)b) \Big) c \bigg) \bmdelta_{rst} \nonumber \\
& = \alpha_{rs} \bigg( \alpha_{s^{\ast}} \Big(\alpha_{r^{\ast}}(a)b \Big)c \bigg) \bmdelta_{rst}.
\end{align}
Additionally, computing the right hand side of (\ref{FellBundle}.\ref{Associativity}), we have:
\begin{align}\label{righthandside}
\mu_{r,st} \bigg( a \bmdelta_r \otimes \mu_{s,t}( b \bmdelta_s \otimes c \bmdelta_t) \bigg)
& = \mu_{r,st} \bigg(
a \bmdelta_r \otimes \alpha_s \Big( \alpha_{s^{\ast}} (b) c \Big) \bmdelta_{st} \bigg) \nonumber \\
& = \alpha_r \bigg( \alpha_{r^{\ast}}(a) \alpha_s \Big( \alpha_{s^{\ast}}(b)c \Big) \bigg) \bmdelta_{rst}.
\end{align}
By these computations, wee see that (\ref{FellBundle}.\ref{Associativity}) holds if and only if
\begin{equation}\label{condition1}
\alpha_{rs} \bigg( \alpha_{s^{\ast}} \Big(\alpha_{r^{\ast}}(a)b \Big)c \bigg)
=
\alpha_r \bigg( \alpha_{r^{\ast}}(a) \alpha_s \Big( \alpha_{s^{\ast}}(b)c \Big) \bigg).
\end{equation}
Therefore, up to applying $\alpha_{r^{\ast}}$ in both sides of (\ref{condition1}), we have proven:

\begin{lemma}\label{Condition}
	A necessary and sufficient condition for the triple
	\begin{equation*}
	\mathcal{B}^{\alpha} = \left(
	\left\{ B_s \right\}_{s \in S},
	\left\{ \mu_{s,t} \right\}_{s,t \in S},
	\left\{ j_{t,s} \right\}_{s,t \in S, s \leq t}
	\right),
	\end{equation*}
	as defined in (\ref{Triple}), to satisfy axiom (\ref{FellBundle}.\ref{Associativity}) is that
	the equality 
	\begin{equation}\label{Equalitycondition}
	\alpha_{s} \bigg( \alpha_{s^{\ast}} \big(\alpha_{r^{\ast}}(a)b \big)c \bigg)
	=
	\alpha_{r^{\ast}}(a) \alpha_s \Big( \alpha_{s^{\ast}}(b)c \Big)
	\end{equation}
	holds for all $a \in A_{rr^{\ast}}$, $b \in A_{ss^{\ast}}$ and $c \in A_{tt^{\ast}}$,
	with $r,s,t\in S$.
\end{lemma}

We will know exploit sufficient conditions on the ideals $A_{ss^{\ast}}$ in order to the triple
$\Triple$ to satisfy (\ref{FellBundle}.\ref{Associativity}).

\begin{proposition}\label{Idempotent1}
	Given an action of an inverse semigroup $S$ on an algebra $A$, a sufficient condition
	for the triple $\Triple$ as defined in (\ref{Triple}) to satisfy (\ref{FellBundle}.\ref{Associativity})
	is that, for each $s \in S$, the ideal $A_{ss^{\ast}}$ is idempotent.
\end{proposition}

\begin{proof}
Fix $s \in S$ and assume $A_{ss^{\ast}}$ is idempotent.
For $r,t\in S$, let $a \in A_{rr^{\ast}}$, $b_1,b_2 \in A_{ss^{\ast}}$, $c \in A_{tt^{\ast}}$
and $b=b_1 b_2$. Notice that
\begin{align*}
\alpha_{s} \bigg( \alpha_{s^{\ast}} \big(\alpha_{r^{\ast}}(a)b \big) c \bigg)
& = \alpha_{s} \bigg( \alpha_{s^{\ast}} \big(\alpha_{r^{\ast}}(a) b_1 b_2 \big) c \bigg)
= \alpha_{r^{\ast}}(a) b_1 \alpha_{s} \Big(\alpha_{s^{\ast}}(b_2) c \Big) \\
& = \alpha_{r^{\ast}}(a) \alpha_{s} \Big( \alpha_{s^{\ast}}(b_1 b_2) c \Big)
= \alpha_{r^{\ast}}(a) \alpha_{s} \Big( \alpha_{s^{\ast}}(b) c \Big).
\end{align*}
Since every element of $A_{ss^{\ast}}$ is a sum of terms of the form $b_1 b_2$, we verify
the equality (\ref{Equalitycondition}) and, hence, by Lemma (\ref{Condition}), we conclude
(\ref{FellBundle}.\ref{Associativity}).
\end{proof}

Finally, there is only one more axiom to worry about in our construction, namely
(\ref{FellBundle}.\ref{NoPatology}).
As already mentioned, this may not hold as well.

Let $a,c \in B_s$ and $b \in B_{s^{\ast}}$ and notice, by the computation made in (\ref{lefthandside}),
that
\begin{align*}
\mu_{ss^{\ast}, s}\bigg(\mu_{s, s^{\ast}}(a \bmdelta_s \otimes b \bmdelta_{s^{\ast}}) \otimes c \bmdelta_s \bigg) 
= \alpha_{ss^{\ast}} \bigg( \alpha_s \Big(\alpha_{s^{\ast}}(a)b \Big)c \bigg) \bmdelta_s
= a \alpha_s(b) c \bmdelta_s.
\end{align*}
So, we immediately have:

\begin{lemma}\label{Condition2}
	A necessary and sufficient condition for the triple
	$$\Triple,$$
	as defined in (\ref{Triple}), to satisfy axiom (\ref{FellBundle}.\ref{NoPatology})
	is that for each $s \in S$ the ideal $A_{ss^{\ast}}$
	satisfies
	$$\Span{A_{ss^{\ast}} A_{ss^{\ast}} A_{ss^{\ast}}} = A_{ss^{\ast}}.$$
\end{lemma}

Notice that, since
$\Span{A_{ss^{\ast}} A_{ss^{\ast}} A_{ss^{\ast}}}
\subseteq \Span{A_{ss^{\ast}} A_{ss^{\ast}}} \subseteq A_{ss^{\ast}}$,
the equality in the Lemma above is equivalent
to $A_{ss^{\ast}}$ being idempotent.
This, combined with Proposition (\ref{Idempotent1}), leads to the following result:

\begin{theorem}\label{sdproductbundle}
	Given an action of an inverse semigroup $S$ on an algebra $A$,
	the triple
	$$\Triple,$$
	as defined in (\ref{Triple}), is a Fell bundle over $S$ if and only if, for each $s \in S$,
	the ideal $A_{ss^{\ast}}$ is idempotent. In this case, it will be henceforth called the
	\emph{semi-direct product bundle}
	relative to the system $(\alpha, S, A)$.
\end{theorem}

\begin{definition}\label{StandingNotation}
	Let $\alpha$ be an action of an inverse semigroup $S$ on an algebra $A$
	satisfying the equivalent conditions of Theorem (\ref{sdproductbundle}). The \emph{crossed product}
	algebra $A \rtimes_{\alpha} S$ is defined to be the cross-sectional
	algebra of the semi-direct product bundle
	$\mathcal{B}^{\alpha}$ associated with $(\alpha, S, A)$.
	We shall denote by $\mathcal{J}_{\alpha}$ the ideal defined
	in (\ref{RedundanceIdeal}) and $q_{\mathcal{J}_{\alpha}}$ the
	quotient map from $\mathcal{L}(\mathcal{B}^{\alpha})$
	to $A \rtimes_{\alpha} S$.
\end{definition}

\begin{remark}\label{RemarkStandingNotation}
	We shall use the notation $a \bmdelta_s$ to denote an element
	of $\mathcal{L}(\mathcal{B}^{\alpha})$ as well, instead of the
	awkward double notation $a \bmdelta_s \bmdelta_s$. The context
	should bring no confusion. Furthermore, according to Remark
	(\ref{remarkDelta}), we shall use the notation $a \Delta_s$ to
	denote the class of $a \bmdelta_s$ in $A \rtimes_{\alpha} S$.
\end{remark}

Notice that, when it is possible to construct the semi-direct product bundle
related to the action of an inverse semigroup on an algebra, the
cross-sectional algebra of the semi-direct product bundle coincides with the
definitions of crossed product algebras in \cite{Beuter} and \cite{Hazrat2018}.
Furthermore, notice still that the cross-sectional construction always leads
to an associative $K$-algebra. Indeed, the problems faced in the construction
of the semi-direct product bundle are related to the problems one would face
to show that the crossed product algebra definitions in \cite{Beuter} and
\cite{Hazrat2018} are associative.


\subsection{Crossed products of interest}

A special case of actions of inverse semigroups on algebras that leads to
a Fell bundle will be investigated now and will be of high interest for
us from now on.

\begin{definition}\label{topaction}
	Let $S$ be an inverse semigroup and let $X$ be a locally compact and
	Hausdorff topological space.
	An \emph{action} of $S$ on $X$ is a semigroup homomorphism
	\begin{equation*}
	\theta: S \to \mathcal{I}(X)
	\end{equation*}
	such that,
	\begin{enumerate}[(i)]
		\item for every $s \in S$, $\theta_s$ is continuous and its domain is
		open in $X$;
		\item\label{nondegenerate} the union of the domains of all the $\theta_s$
		coincides with $X$.
	\end{enumerate}
	The triple $(\theta, S, X)$ is called a \emph{(topological) dynamical system}.
	Furthermore, if $X$ is totally disconnected and the domains are \emph{clopen}
	(closed and open), we say that $\theta$ is an
	\emph{ample action} and $(\theta, S, X)$
	is an \emph{ample dynamical system}.
\end{definition}

The reader is invited to compare the definition of ample action
above with the Definition 4.2 of ample action in \cite{Steinberg}.
Is his definition, Steinberg requires the domains to be
compact-open instead of just clopen. As we shall see, there is
no need to require the domains to be compact in order to
obtain an ample groupoid of germs.

For every $e \in E(S)$, we denote by $X_e$ the domain of
$\theta_e$. Therefore,
for each $s \in S$, $\theta_s$ is a homeomorphism
from $\dom{s}$ to $\ran{s}$.

\textbf{From now on, we fix a Hausdorff, locally compact,
totally disconnected topological
space $X$ and an ample system $(\theta, S, X)$.}

We will henceforth denote by $\Lc{X}$ the set of all locally constant,
compactly supported, $K$-valued
functions on $X$ and denote by $\supp(f)$ the support of $f \in \Lc{X}$. With pointwise
multiplication, $\Lc{X}$ is a commutative $K$-algebra, which is unital
if and only if $X$ is compact.

For each $s \in S$, we may also consider the $K$-algebra $\Lc{\dom{s}}$,
which we will identify with the set formed by all $f \in \Lc{X}$ vanishing on
$X \setminus \dom{s}$. Under this identification $\Lc{\dom{s}}$ becomes an ideal
in $\Lc{X}$.

So, from the ample action of $S$ on a locally compact Hausdorff
space $X$ in the sense
of (\ref{topaction}), it is easy to construct an action of $S$
on $\Lc{X}$ in the
sense of (\ref{algaction}). Regarding the homeomorphism
$\theta_s: \dom{s} \to \ran{s}$, we may define an isomorphism
\begin{equation*}
\alpha_s: \Lc{\dom{s}} \to \Lc{\ran{s}},
\end{equation*}
by setting
\begin{equation}\label{inducedaction}
\alpha_s(f) = f \circ \theta_{s^{\ast}},
\end{equation}
for all $f \in \Lc{\dom{s}}$.
This said, $\alpha: S \to \mathcal{I}(\Lc{X})$ is a semigroup homomorphism
which is then easily seen to be an (algebraic) action of $S$ on $\Lc{X}$.

Furthermore, it is clear that, for every $s \in S$, the ideal $\Lc{\ran{s}}$
is idempotent. So, we may now construct the semi-direct product bundle
associated with $(\alpha,S,\Lc{X})$, which we will denote by
$\mathcal{B}^{\theta}$ in this case. This leads to the cross-sectional
algebra $\Lcs$, which is our object of interest.

Notice that, since we are assuming that $X_e$ is clopen for every $e \in E(S)$,
its characteristic function $1_{X_e}$ is locally constant, but not necessarily
compactly supported. However, for any $f \in \Lc{X}$, the product $f1_{X_e}$
is compactly supported and so, it lies in $\Lc{X_e}$. Then, we may globally
define an endomorphism $\bar{\alpha}_s: \Lc{X} \to \Lc{X}$ given by
$$\bar{\alpha}_s(f):=\alpha_s(f1_{s^\ast s})$$
where $1_{s s^{\ast}}$ stands for $1_{X_{ss^\ast}}$.

In this context, notice that, for any $s,t \in S$, $f \in \Lc{\ran{s}}$ and
$g \in \Lc{\ran{t}}$
$$
\alpha_s(\alpha_{s^\ast}(f)g)
= \alpha_s(\alpha_{s^\ast}(f)1_{s^\ast s}g)
= \alpha_s(\alpha_{s^\ast}(f)) \alpha_s(1_{s^\ast s}g)
= f \bar{\alpha}_s(g).
$$
Hence, we get a simpler formula for the product
$$f \Delta_s \cdot g \Delta_t = f \bar{\alpha}_s(g) \Delta_{st}.$$

Now we begin the preparations to obtain an universal property for $\Lcs$.

\begin{definition}\label{covrep}
	A \emph{covariant representation}
	of the system $(\theta, S, X)$ on a
	$K$-vector space $V$ is a pair $(\pi, \sigma)$, where
	$\pi: \Lc{X} \to L(V)$ is a non-degenerate representation and
	$\sigma: S \to L(V)$ is a semigroup homomorphism such that:
	\begin{enumerate}[(i)]
		\item \label{identity} $\pi(\alpha_s(f)) = \sigma_s \pi(f) \sigma_{s^{\ast}}$ for
		$s \in S$ and $f \in \Lc{\dom{s}}$;
		\item \label{essentialspace} $\Span \Set{\pi(f)\xi}{f \in \Lc{X_e}, \xi \in V}=\sigma_e(V)$
		for every $e \in E(S)$.
	\end{enumerate} 
\end{definition}

Notice that, if the domains of all $\theta_s$ are compact, then condition
(\ref{essentialspace}) of Definition (\ref{covrep}) may be replaced
equivalently by:

\begin{enumerate}[(ii')]
	\item $\pi(1_{X_e})=\sigma_e$ for $e \in E(S)$.
\end{enumerate}

The next lemma is a very helpful tool

\begin{lemma}\label{helpful}
	Let $(\pi,\sigma)$ be a covariant representation of
	the system $(\theta,S,X)$
	on a vector space $V$. If $f \in \Lc{X_e}$ for some $e \in E(S)$, then
	$$\sigma_e\pi(f)=\pi(f) = \pi(f) \sigma_e.$$
\end{lemma}

\begin{proof}
Let $\xi \in V$. By (\ref{covrep}.\ref{essentialspace}), there exist $\eta \in V$
such that $\pi(f)\xi = \sigma_e(\eta)$. Notice that
$$\sigma_e \pi(f) \xi = \sigma_e(\sigma_e(\eta)) = \sigma_e(\eta) = \pi(f)\xi.$$
Hence, $\sigma_e\pi(f) = \pi (f)$ and the first equality is proved.
For the second, observe that
$$
\pi(f) \sigma_e
= \sigma_e \pi(f) \sigma_e
\overset{(\ref{covrep}.\ref{identity})}{=} \pi(\alpha_e(f)) = \pi(f),
$$
proving the second equality.
\end{proof}

The next proposition shows that covariant representations of the system can
be integrated to non-degenerate representations of the semi-direct product
bundle associated and, hence, to a non-degenerate representation of the
crossed product algebra.

\begin{proposition}\label{integration}
	Let $(\pi,\sigma)$ be a covariant representation of the system $(\theta,S,X)$
	on a vector space $V$. For each $s \in S$, consider the map $\pi_s: B_s \to L(V)$
	given by $\pi_s(f \bmdelta_s) = \pi(f) \sigma_s$ for $f \in \Lc{\ran{s}}$.
	Then, the collection $\Pi=\left\{ \pi_s \right\}_{s \in S}$ is a non-degenerate
	representation of the semi-direct product bundle
	$\mathcal{B}^{\theta}$ on $V$.
\end{proposition}

\begin{proof}
For each $s \in S$, the map $\pi_s$ is clearly linear and, if $s,t \in S$, $f \in \Lc{\ran{s}}$
and $g \in \Lc{\ran{t}}$, we have
\begin{align*}
\pi_{st} \Big( \mu_{s,t}(f \bmdelta_s \otimes g \bmdelta_t) \Big)
& = \pi_{st} \Big(\alpha_s \big( \alpha_{s^{\ast}}(f)g \big)\bmdelta_{st} \Big)
= \pi \Big( \alpha_s \big( \alpha_{s^{\ast}}(f)g \big) \Big) \sigma_{st} \\
& \overset{(\ref{covrep}.\ref{identity})}{=}
		\sigma_s \pi \big( \alpha_{s^{\ast}}(f)g \big) \sigma_{s^{\ast}} \sigma_{st}
\overset{(\ref{helpful})}{=}
		\sigma_s \pi \big( \alpha_{s^{\ast}}(f) \big) \pi(g) \sigma_t \\
& \overset{(\ref{covrep}.\ref{identity})}{=} \sigma_s \sigma_{s^{\ast}} \pi(f) \sigma_s \pi(g)\sigma_t
\overset{(\ref{helpful})}{=} \pi(f) \sigma_s \pi(g) \sigma_t \\
& = \pi_s(f \bmdelta_s) \pi_t(g \bmdelta_t).
\end{align*}
Furthermore, if $s \leq t$, then
$$
\pi_t(j_{t,s}(f \bmdelta_s)) = \pi_t(f \bmdelta_t) = \pi(f) \sigma_t
\overset{(\ref{helpful})}{=} \pi(f) \sigma_{ss^{\ast}} \sigma_t
= \pi(f) \sigma_s = \pi_s(f \bmdelta_s),
$$
concluding that $\Pi$ is indeed a representation of
$\mathcal{B}^{\theta}$ on $V$.

Finally, let $\xi \in V$ and write
$$\xi = \sum_{i=1}^n \pi(f_i) \xi_i,$$
by the non-degenerateness of $\pi$.
Since $\Lc{X} = \sum_{e \in E(S)} \Lc{X_e}$, we may assume that each $f_i$
lies in $\Lc{X_{e_i}}$ for some $e_i \in E(S)$.
Hence,
$$
\xi
= \sum_{i=1}^n \pi(f_i) \xi_i
\overset{(\ref{helpful})}{=} \sum_{i=1}^n \pi(f_i) \sigma_{e_i} (\xi_i)
= \sum_{i=1}^n \pi_{e_i}(f_i \bmdelta_{e_i}) \xi_i,
$$
concluding the proof.
\end{proof}

Combining this result with (\ref{UniversalProp}), we get:

\begin{corollary}\label{integration2}
	Let $(\pi,\sigma)$ be a covariant representation of the system $(\theta,S,X)$
	on a vector space $V$. Then, there exists a non-degenerate
	representation
	$\pi \times \sigma: \Lcs \to L(V)$ such that
	$(\pi \times \sigma)(f \Delta_s) = \pi(f) \sigma_s$ for
	$f \in \Lc{\ran{s}}$.
\end{corollary}

Now we proceed the other way around. The goal is to prove that every non-degenerate
representation $\Pi$ of the semi-direct product bundle on a vector space $V$ is given
as above for a covariant representation $(\pi,\sigma)$ of $(\theta,S,X)$.
We begin with the following lemma.

\begin{lemma}\label{disintegration1}
	Given a non-degenerate representation $\Pi = \left\{ \pi_s \right\}_{s \in S}$
	of the  semi-direct
	product bundle $\mathcal{B}^{\theta}$ on a vector space $V$,
	there exists a non-degenerate representation $\pi$ of $\Lc{X}$ on $V$
	such that
	$$\pi(f) = \pi_e(f \bmdelta_e)$$
	for all $e \in E(S)$ and all $f \in \Lc{X_e}$.
\end{lemma}

\begin{proof}
Let $f \in \Lc{X}$. Since $\Lc{X}=\sum_{e \in E(S)} \Lc{X_e}$, we may write it as
finite sum
$f = \sum_{e \in E(S)} f_e$ where $f_e \in \Lc{X_e}$.
We claim initially that $\sum_{e \in E(S)} \pi_e(f_e \bmdelta_e)$ vanishes when $f=0$.
In fact, since $\Pi$ is non-degenerate, it is enough to prove that
$$\sum_{e \in E(S)} \pi_e(f_e \bmdelta_e) \pi_s(g \bmdelta_s) = 0$$
for all $s \in S$ and $g \in \Lc{\ran{s}}$. Notice that
\begin{align*}
\sum_{e \in E(S)} \pi_e(f_e \bmdelta_e) \pi_s(g \bmdelta_s)
& = \sum_{e \in E(S)} \pi_{es}(f_e g \bmdelta_{es})
\overset{es \leq s}{=} \sum_{e \in E(S)} \pi_s(f_e g \bmdelta_s) \\
& = \pi_s \bigg( \sum_{e \in E(S)} f_e g \bmdelta_s \bigg)
= \pi_s (fg \bmdelta_s)=0,
\end{align*}
proving the claim. Hence, the map $\pi: \Lc{X} \to L(V)$ defined by
$$ \pi (f) = \sum_{e \in E(S)} \pi_e(f_e \bmdelta_e)$$
does not depend of the choice of the $f_e$'s.
Furthermore, notice that, for $e,e' \in E(S)$, $f \in \Lc{X_e}$ and $g \in \Lc{X_{e'}}$,
we have
$$
\pi(f)\pi(g) = \pi_e(f \bmdelta_e) \pi_{e'}(g \bmdelta_{e'})
= \pi_{e e'}(fg \bmdelta_{ee'}) = \pi(fg).
$$
By linearity, $\pi$ is a representation of $\Lc{X}$
on $V$.

Finally, the non-degenerateness of $\pi$ comes from the non-degenerateness of $\Pi$.
In fact, let $s \in S$, $f \in \Lc{\ran{s}}$ and
choose $g, h \in \Lc{\ran{s}}$ such that $f=gh$. Hence
$$
\pi_s( f \bmdelta_s)
= \pi_{ss^{\ast}} (g \bmdelta_{ss^{\ast}}) \pi_s(h \bmdelta_s)
= \pi(g) \pi_s(h \bmdelta_s),
$$
concluding the argument.
\end{proof}

With this in hands, we proceed to the promised result.

\begin{theorem}\label{disintegration}
	Given a non-degenerate representation
	$\Pi = \left\{ \pi_s \right\}_{s\in S}$ of
	the semi-direct product bundle $\mathcal{B}^{\theta}$ on a vector
	space $V$, there exists a covariant representation $(\pi,\sigma)$
	of the system
	$(\theta,S,X)$ on $V$ such that
	$$\pi_s(f \bmdelta_s) = \pi(f) \sigma_s$$
	for every $s \in S$ and $f \in \Lc{\ran{s}}$.
\end{theorem}

\begin{proof}
Let $\pi$ be the representation of $\Lc{X}$ on $V$ given in the Lemma (\ref{disintegration1}).
Since $\pi$ is non-degenerate, given any $\xi \in V$, we may write
$$\xi = \sum_{i=1}^n \pi(f_i) \xi_i,$$
where each $f_i \in \Lc{X}$ and $\xi_i \in V$.
We then define, for each $s \in S$,
$$\sigma_s(\xi) = \sum_{i=1}^n \pi_s(\bar{\alpha}_s(f_i) \bmdelta_s) \xi_i.$$
To prove that $\sigma_s$ is well defined, we must show that
the right hand side of the equality above vanishes when $\xi=0$.
Hence, suppose $\xi=0$ and let
$$C= \bigcup_{i=1}^n \supp (f_i) \cap \dom{s}.$$
So $C$ is a compact open set and, for each $i=1, \ldots, n$, we have
$1_C f_i 1_{s^\ast s} = f_i 1_{s^\ast s}$.
Therefore,
\begin{align*}
\sum_{i=1}^n \pi_s(\bar{\alpha}_s(f_i) \bmdelta_s) \xi_i
& = \sum_{i=1}^n \pi_s(\alpha_s(1_C f_i 1_{s^\ast s}) \bmdelta_s) \xi_i
= \sum_{i=1}^n \pi_s(1_{\theta_s(C)} \bmdelta_s) \pi_{s^{\ast}s}(f_i 1_{s^\ast s} \bmdelta_{s^{\ast}s}) \xi_i \\
& = \sum_{i=1}^n \pi_s(1_{\theta_s(C)} \bmdelta_s) \pi(f_i 1_{s^\ast s}) \xi_i
= \sum_{i=1}^n \pi_s(1_{\theta_s(C)} \bmdelta_s) \pi(1_C f_i 1_{s^\ast s}) \xi_i \\
& = \pi_s(1_{\theta_s(C)} \bmdelta_s) \pi(1_C 1_{s^\ast s}) \sum_{i=1}^n \pi(f_i)\xi_i
= \pi_s(1_{\theta_s(C)} \bmdelta_s) \pi(1_C 1_{s^\ast s}) \xi = 0,
\end{align*}
concluding that $\sigma_s$ is well defined.
Furthermore, $\sigma:S \to L(V)$ given by $s \mapsto \sigma_s$ is a semigroup homomorphism.
In fact, consider $s,t \in V$ and $\xi \in V$ a vector of the form
$\xi = \pi(\varphi) \eta$ for some $\varphi \in \Lc{X}$
and $\eta \in V$. Additionally, consider $f,g \in \Lc{\ran{t}}$ such that $fg=\bar{\alpha}_t(\varphi)$
and observe that
\begin{align*}
\sigma_s \sigma_t(\xi)
& = \sigma_s \sigma_t \pi(\varphi) \eta
= \sigma_s \pi_t (\bar{\alpha}_t(\varphi) \bmdelta_t) \eta
= \sigma_s \pi_{tt^\ast}(f \bmdelta_{tt^\ast}) \pi_t(g \bmdelta_t) \eta \\
& = \sigma_s \pi(f) \pi_t(g \bmdelta_t) \eta
= \pi_s(\bar{\alpha}_s(f)\bmdelta_s) \pi_t (g \bmdelta_t) \eta
= \pi_{st}(\bar{\alpha}_s(f) \bar{\alpha}_s(g) \bmdelta_{st}) \eta \\
& = \pi_{st}(\bar{\alpha}_s(\bar{\alpha}_t(\varphi)) \bmdelta_{st}) \eta
= \pi_{st}(\bar{\alpha}_{st}(\varphi) \bmdelta_{st}) \eta
= \sigma_{st}(\xi).
\end{align*}
Since the set of vectors of the form $\xi = \pi(\varphi) \eta$ spans $V$, we conclude that
$\sigma$ is a semigroup homomorphism.
With the aim of proving (\ref{covrep}.\ref{identity}), let $\xi = \pi(\varphi) \eta$ as above
and notice that	
\begin{align*}
\sigma_s \pi(f) \sigma_{s^\ast}(\xi)
& = \sigma_s \pi(f) \sigma_{s^\ast} \pi(\varphi) \eta
= \sigma_s \pi(f) \pi_{s^\ast}(\bar{\alpha}_{s^\ast}(\varphi) \bmdelta_{s^\ast}) \eta \\
& = \pi_s(\alpha_s(f) \bmdelta_s) \pi_{s^\ast}(\bar{\alpha}_{s^\ast}(\varphi) \bmdelta_{s^\ast}) \eta
= \pi_{ss^\ast}(\alpha_s(f) \varphi \bmdelta_{ss^\ast}) \eta \\
& = \pi(\alpha_s(f) \varphi) \eta
= \pi(\alpha_s(f)) \xi.
\end{align*}
For the proof of (\ref{covrep}.\ref{essentialspace}), fix $e \in E(S)$ and let
$f \in \Lc{X_e}$, $\xi \in V$ and notice that
$$\pi(f)\xi = \pi_e(f \bmdelta_e) \xi = \sigma_e(\pi(f) \xi)$$
and hence the essential space of $\pi(\Lc{X_e})$ is contained
in the range of $\sigma_e$, that is
$\Span \Set{\pi(f)\xi}{f \in \Lc{X_e}, \xi \in V} \subseteq \sigma_e(V)$.
Conversely, let $\xi \in \sigma_e(V)$. Then, there exists $\eta \in V$ such that
$\sigma_e(\eta) = \xi$. Since $\pi$ is non-degenerate, there exists $f_i \in \Lc{X}$
and $\eta_i \in V$ such that $\eta=\sum_{i=1}^{n} \pi(f_i) \eta_i$. Therefore
$$
\xi
= \sigma_e(\eta)
= \sum_{i=1}^{n} \pi_e(\bar{\alpha}_e(f_i)\bmdelta_e) \eta_i
= \sum_{i=1}^{n} \pi(\bar{\alpha}_e(f_i)) \eta_i,
$$
from where we conclude that the range of $\sigma_e$ is contained in the essential space of
$\pi(\Lc{X_e})$.

Finally, we must prove that $\pi_s(f \bmdelta_s) = \pi(f) \sigma_s$
for every $s \in S$ and $f \in \Lc{\ran{s}}$.
For this task, let $\xi \in V$ and observe that
$$
\pi_s(f \bmdelta_s) \xi
= \sigma_s \pi (\alpha_{s^\ast}(f)) \xi
= \sigma_s \sigma_{s^\ast} \pi(f) \sigma_s (\xi)
\overset{(\ref{helpful})}{=} \pi(f) \sigma_s(\xi),
$$
concluding the proof.
\end{proof}

An immediate result about disintegration of representations of
the crossed product algebra follows. 

\begin{corollary}\label{disintegration0}
	Given a non-degenerate representation $\Phi$ of the
	crossed product $\Lcs$
	on a vector space $V$, there exists a unique covariant representation $(\pi,\sigma)$ of the system $(\theta,S,X)$ on $V$
	such that
	$$\Phi(f \Delta_s) = \pi(f) \sigma_s$$
	for every $s \in S$ and $f \in \Lc{\ran{s}}$.
\end{corollary}

The next lemma is another helpful result.

\begin{lemma}\label{helpful2}
	Let $(\pi,\sigma)$ be a covariant representation of the
	system $(\theta,S,X)$
	on a vector space $V$. Then,
	$$\sigma_s\pi(f)=\pi( \bar{\alpha}_s(f))\sigma_s$$
	for any $s \in S$ and $f \in \Lc{X}$.
\end{lemma}

\begin{proof}
	In the presence of (\ref{integration}) and (\ref{disintegration}),
	we may assume that $\pi$ and $\sigma$ are given as in the
	proof of (\ref{disintegration}) for the non-degenerate
	representation of the semi-direct product bundle
	$\mathcal{B}^{\theta}$ given as in (\ref{integration}). Hence,
	let $\xi \in V$ be a vector of the form $\xi = \pi(\varphi)\eta$
	for some $\varphi \in \Lc{X}$ and $\eta \in V$ and notice that
	\begin{align*}
		\sigma_s \pi(f) \xi
		& = \sigma_s \pi(f) \pi(\varphi) \eta
		= \sigma_s \pi(f \varphi) \eta
		=\pi_s( \balpha_s(f \varphi) \bmdelta_s) \eta \\
		& = \pi_{ss^{\ast}}(\balpha_s(f) \bmdelta_{ss^{\ast}})
				\pi_s(\balpha_s(\varphi) \bmdelta_s) \eta
		=\pi(\balpha_s(f)) \sigma_s( \pi(\varphi)\eta)
		= \pi(\balpha_s(f)) \sigma_s \xi
	\end{align*}
	Since the vectors of the above form spans $V$, by linearity
	we conclude the proof.
\end{proof}

Since we are discussing representations of $\Lcs$, we shall see now that,
for any ideal $J$ of $\Lcs$, there exists a non-degenerate representation
whose kernel coincides with $J$. For that, we resort to
Proposition (5.1) of
\cite{DokuchaevExel}.

\begin{proposition}\label{nondegrep}
	Let $A$ be a $K$-algebra possessing local units
	\footnote{Recall that $A$ is said to have local units if, for every
		$a$ in $A$, there exists an idempotent $e \in A$, such that
		$ea = a = ae$.}.
	Then, for every ideal
	$J \trianglelefteq A$, there exists a vector space $V$ and a non-degenerate representation
	$$\pi: A \to L(V),$$
	such that $J = \ker(\pi)$.
\end{proposition}

To see that that above result applies to our situation, we give the following:

\begin{proposition}\label{localunits}
	$\Lcs$ has local units.
\end{proposition}

\begin{proof}
	Let $b=\sum_{s \in F} f_s \Delta_s$, with $F \subseteq S$ finite.
	For each $s \in S$, let $C_s = \supp(f_s) \subseteq \ran{s}$. Then,
	$C_s$ is a compact open set of $X$, as well as
	$\theta_{s^{\ast}}(C_s) \subseteq \dom{s}$.
	Therefore,
	$$E:= \bigcup_{s \in F} \bigg( C_s \cup \theta_{s^{\ast}}(C_s) \bigg),$$
	is also compact open in $X$, since it is a finite union of compact
	open sets.
	
	For each $\varepsilon \in \mathcal{P}(F)$, define
	$$
	C_{\varepsilon}
	:= \bigg( \bigcap_{s \in \varepsilon} C_s\bigg)
	\cap \bigg( \bigcap_{s \in F \setminus \varepsilon} C_s^c\bigg)
	\quad \text{ and } \quad
	D_{\varepsilon}
	:= \bigg( \bigcap_{s \in \varepsilon} \theta_{s^{\ast}}(C_s) \bigg)
	\cap \bigg( \bigcap_{s \in F \setminus \varepsilon} \theta_{s^{\ast}}(C_s)^c\bigg),
	$$
	which are compact open sets as well, since they are closed sets contained
	in a compact set and $X$ is Hausdorff.
	
	Now, for every pair $(\varepsilon, \varsigma) \in
	\mathcal{P}(F) \times \mathcal{P}(F)$, such that
	$|\varepsilon|+|\varsigma|>0$, define
	$$
	E_{(\varepsilon, \varsigma)}
	:= C_{\varepsilon} \cap D_{\varsigma}
	\quad \text{ and } \quad
	e_{(\varepsilon, \varsigma)}
	:= \bigg( \prod_{s \in \varepsilon} ss^{\ast} \bigg)
	\bigg( \prod_{s \in \varsigma} s^{\ast}s \bigg).
	$$
	Since $E_{(\varepsilon, \varsigma)}$ is compact open,
	$\varphi_{(\varepsilon, \varsigma)}
	:= 1_{E_{(\varepsilon, \varsigma)}} \Delta_{e_{(\varepsilon, \varsigma)}}$
	is a well defined element of $\Lcs$. Moreover,
	$\left\{ \varphi_{(\varepsilon, \varsigma)} \right\}_{|\varepsilon|+|\varsigma|>0}$
	is a collection of orthogonal idempotent elements.
	
	Hence,
	$$
	\varphi
	= \sum_{|\varepsilon|+|\varsigma|>0} \varphi_{(\varepsilon, \varsigma)}
	$$
	is an idempotent element of $\Lcs$ and notice that
	\begin{align*}
	\varphi \bigg( \sum_{s \in F} f_s \Delta_s \bigg)
	& = \sum_{s \in F} \bigg( \sum_{|\varepsilon|+|\varsigma|>0}
	\varphi_{(\varepsilon, \varsigma)} \bigg) f_s \Delta_s
	= \sum_{s \in F} \sum_{|\varepsilon|+|\varsigma|>0}
	1_{E_{(\varepsilon, \varsigma)}} f_s \Delta_{e_{(\varepsilon, \varsigma)}s}\\
	& = \sum_{s \in F} \sum_{|\varepsilon|+|\varsigma|>0}
	1_{E_{(\varepsilon, \varsigma)}} f_s \Delta_s
	= \sum_{s \in F} 1_E f_s \Delta_s
	=  \sum_{s \in F} f_s \Delta_s
	\end{align*}
	and
	\begin{align*}
	\bigg( \sum_{s \in F} f_s \Delta_s \bigg) \varphi
	& = \sum_{s \in F} \sum_{|\varepsilon|+|\varsigma|>0} f_s \Delta_s
	\varphi_{(\varepsilon, \varsigma)}
	= \sum_{s \in F} \sum_{|\varepsilon|+|\varsigma|>0} f_s
	\balpha_s(1_{E_{(\varepsilon, \varsigma)}}) \Delta_{se_{(\varepsilon, \varsigma)}}\\
	& = \sum_{s \in F} \sum_{|\varepsilon|+|\varsigma|>0} f_s
	\balpha_s(1_{E_{(\varepsilon, \varsigma)}}) \Delta_s
	= \sum_{s \in F} f_s \balpha_s(1_E) \Delta_s \\
	& =  \sum_{s \in F} f_s 1_{\theta_s(E \cap \dom{s})} \Delta_s
	= \sum_{s \in F} f_s \Delta_s,
	\end{align*}
	where the last equality holds because $C_s \subseteq \theta_s(E \cap \dom{s})$.
\end{proof}

Joining this two results, we immediately have.

\begin{corollary}\label{AllIdealsReps}
	For every ideal $J \trianglelefteq \Lcs$,
	there exists a vector space $V$ and a non-degenerate representation
	$$\pi: \Lcs \to L(V),$$
	such that $J = \ker(\pi)$.
	In particular, $\Lcs$ has a faithful non-degenerate representation.
\end{corollary}

With this result in hand, we shall present two interesting consequences for $\Lcs$.

\begin{proposition}\label{LcImmerse}
	Theres is a monomorphism $\phi: \Lc{X} \to \Lcs$ such that
	$$\phi(f)=f \Delta_e,$$
	whenever $e \in E(S)$ and $f \in \Lc{X_e}$.
\end{proposition}

\begin{proof}
	Since $\Lc{X} = \sum_{e \in E(S)} \Lc{X_e}$, we may write any function
	$f \in \Lc{X}$ as a finite sum
	$$f=\sum_{e \in E(S)} f_e,$$
	with $f_e \in \Lc{X_e}$.
	By (\ref{AllIdealsReps}), $\Lcs$ has a faithful non-degenerate representation,
	which is the integrated form $\pi \times \sigma$ of some covariant
	representation $(\pi,\sigma)$ for the system $(\theta, S, X)$, by
	(\ref{disintegration0}).
	
	Notice that,
	\begin{equation*}
		(\pi \times \sigma)\bigg( \sum_{e \in E(S)} f_e \Delta_e \bigg)
		= \sum_{e \in E(S)} \pi(f_e) \sigma_e
		\overset{(\ref{helpful})}{=} \sum_{e \in E(S)} \pi(f_e)
		= \pi(f).
	\end{equation*}
	So, if $f=0$, we must have $\sum_{e \in E(S)} f_e \Delta_e =0$
	by the faithfulness of $\pi \times \sigma$. Hence, the map
	$\phi: \Lc{X} \to \Lcs$ given by
	$$\phi(f) = \sum_{e \in E(S)} f_e \Delta_e$$
	is well defined and it is also a homomorphism.
	
	Furthermore, $\phi$ is injective. Indeed, if $\phi(f)=0$, then
	$\pi(f)=0$. Let $x \in X$ and choose $e \in E(S)$ such that
	$x \in X_e$. Choose $\varphi \in \Lc{X_e}$ such that
	$\varphi(x)=1$. Then,
	$$
	0
	= \pi(f)\pi(\varphi)
	= \pi (f \varphi)
	\overset{(\ref{helpful})}{=} \pi(f \varphi) \sigma_e
	=(\pi \times \sigma)( f \varphi \Delta_e)
	$$
	from where we conclude that $f \varphi=0$ and, hence, $f(x)=0$.
	Since $x$ is arbitrary, $f=0$.
\end{proof}

Relying on this proposition, we may then identify $\Lc{X}$ as a subalgebra
of $\Lcs$. Furthermore, with such an identification, keeping in mind
the definition given in the proof of (\ref{disintegration1}), we may
interpret the map $\pi$ of the covariant representation $(\pi,\sigma)$
obtained by the disintegration of a non-degenerate representation $\phi$
of $\Lcs$ as its restriction to $\Lc{X}$.

The second consequence is an interesting characterization for the
ideal $\mathcal{J}_{\alpha}$ defined in (\ref{StandingNotation}).
Indeed, consider the ideal
\begin{equation}\label{ideal2}
	\mathcal{I}_{\alpha}
	:= \bigcap_{(\pi,\sigma)}
		\ker \bigg(
			(\pi \times \sigma) \circ q_{\mathcal{J}_{\alpha}}
		\bigg)
	\trianglelefteq
	\mathcal{L}(\mathcal{B}^{\theta}),
\end{equation}
in which $(\pi, \sigma)$ ranges over covariant representations of
the system
$(\theta, S, X)$ and
$\pi \times \sigma$ is the integrated form of $(\pi,\sigma)$ for $\Lcs$.

\begin{proposition}
	The ideal $\mathcal{I}_{\alpha}$, as defined in \eqref{ideal2} above,
	coincides with the ideal $\mathcal{J}_{\alpha}$.
\end{proposition}

\begin{proof}
	It is clear that $\mathcal{J}_{\alpha} \subseteq \mathcal{I}_{\alpha}$.
	It remains to prove the reverse inclusion. Indeed, by (\ref{AllIdealsReps}),
	$\Lcs$ has a faithful representation which is of the form $\pi \times \sigma$
	for some covariant representation $(\pi,\sigma)$ of the system
	$(\theta, S, X)$. Hence,
	$$
	\mathcal{I}_{\alpha}
	\subseteq
	\ker \bigg(
	(\pi \times \sigma) \circ q_{\mathcal{J}_{\alpha}}
	\bigg)
	= \mathcal{J}_{\alpha},
	$$
	concluding the proof.
\end{proof}

At first, there is no apparent reason for this result to be valid
in the general case $A \rtimes_{\alpha} S$.

We end this section with a curious fact.
Let $(\theta, S, X)$ be an ample dynamical system as usual all over this
subsection. It is well known that one can always add a
formal unit to a semigroup $S$, leading to the unitization
$$S^{+} := S \sqcup \left\{ 1 \right\}$$
of $S$. One may also extend
$\theta$ to $\theta^{+}: S^{+} \to \mathcal{I}(X)$
in a natural way by defining $\theta_1$ as the identity map on $X$.
It is clear that $(\theta^{+}, S^{+},X)$ is an ample action as well,
which we shall call the unitization of $(\theta, S ,X)$.
We shall denote by $\alpha^{+}$ the action of $S^{+}$ on $\Lc{X}$ induced
by $\theta^{+}$ in the sense of (\ref{inducedaction}). In this context,
we have:

\begin{proposition}
	Let $(\theta, S, X)$ be an ample dynamical system and
	$(\theta^{+}, S^{+}, X)$ its unitization. Then
	$$\Lc{X} \rtimes_{\alpha} S \simeq \Lc{X} \rtimes_{\alpha^{+}} S^{+}.$$
\end{proposition}

\begin{proof}
	Let $\mathcal{B}^{\theta}$ be the semi-direct product
	bundle associated with
	$(\theta, S, X)$. For each $s \in S$, define
	$\pi_s: B_s \to \Lc{X} \rtimes_{\alpha^{+}} S^{+}$
	by $\pi_s(f \bmdelta_s) = f \Delta_s$. It is clear that
	$\Pi = \left\{ \pi_s \right\}_{s \in S}$ is a representation
	of $\mathcal{B}^{\theta}$ in $\Lc{X} \rtimes_{\alpha^{+}} S^{+}$.
	By,
	(\ref{UniversalProp}), there exists a homomorphism
	$$\phi: \Lc{X} \rtimes_{\alpha} S \to \Lc{X} \rtimes_{\alpha^{+}} S^{+}$$
	such that $\phi(f \Delta_s) = f \Delta_s$ for every $s \in S$.
	
	Since every $f \in \Lc{X}$ may be written
	as a finite sum $f = \sum_{e \in E(S)} f_e$, we have
	$$\phi \bigg(\sum_{e \in E(S)} f_e \Delta_e \bigg) = \sum_{e \in E(S)} f_e \Delta_e
	= \sum_{e \in E(S)} f_e \Delta_1 = f \Delta_1.$$
	Thus proving that $\phi$ is onto $\Lc{X} \rtimes_{\alpha^{+}} S^{+}$.
	
	By (\ref{AllIdealsReps}), $\Lc{X} \rtimes_{\alpha} S$ has a
	non-degenerate
	faithful representation which is the integrated form $\pi \times \sigma$
	of some covariant representation $(\pi, \sigma)$ of the system
	$(\theta, S, x)$, by (\ref{disintegration0}).
	We may then extend $\sigma$ to a semigroup
	homomorphism $\sigma^{+}: S^{+} \to L(V)$, by setting $\sigma_1$
	as the identity map on $V$. It is then easy to see that
	$(\pi, \sigma^{+})$ is a covariant representation of the system
	$(\theta^{+}, S^{+}, X)$. By (\ref{integration2}), there exists
	a representation $\pi \times \sigma^{+}$ of
	$\Lc{X} \rtimes_{\alpha^{+}} S^{+}$ on $V$ such that
	$(\pi \times \sigma^{+})(f \Delta_s) = \pi(f) \sigma^{+}_s$
	for every $s \in S^{+}$.
	
	We thus have the following commutative diagram
	\[ \begin{tikzcd}
	\Lc{X} \rtimes_{\alpha} S \arrow{r}{\phi} \arrow[swap]{d}{\pi \times \sigma} &
	\Lc{X} \rtimes_{\alpha^{+}} S^{+} \arrow{dl}{\pi \times \sigma^{+}} \\
	L(V) 
	\end{tikzcd}.
	\]
	Finally, if $b \in \Lc{X} \rtimes_{\alpha} S$ lies in the kernel of $\phi$, then
	it must also lie in the kernel of $\pi \times \sigma$. Hence, $b=0$ and
	$\phi$ is injective.
\end{proof}

%% file: Induction.tex
\section{Induction Process}
\label{induction}

In this section we follow the ideas introduced by Dokuchaev and Exel in
\cite{DokuchaevExel}. They study the ideal structure of algebraic partial
crossed products, in the context of a discrete group acting on a Hausdorff,
locally compact, totally disconnected topological space. We shall study the
ideal structure of the crossed product algebra in our context.

\textbf{Throughout this section, we fix an ample dynamical
system $(\theta, S, X)$.}


\subsection{Induction process}

For any point $x \in X$, as in the case of group actions,
one can speak
about its orbit
$$\orb(x) := \Set{\theta_s(x)}{x \in \dom{s}}.$$

However, by trying to bring the concept of isotropy of a point from the case
of group actions, one should come across the fact that the set
$$\tilde{G}_{x} := \Set{s \in S}{x \in \dom{s}, \, \theta_s(x)=x}$$
does not need to have a group structure at all. In fact, $\tilde{G}_{x}$
as defined above is a $\ast$-subsemigroup of $S$. We shall define the isotropy
group of $x$ as the maximal group image
\footnote{Th maximal group homomorphic image
of an inverse semigroup $S$ is a group $G(S)$ satisfying the following
property:
if $G$ is a group and $\psi: S \to G(S)$ is a surjective homomorphism,
then $\psi$ factors through $G(S)$. See Proposition 2.1.2
of \cite{patersonbook} for
further details.}
of $\tilde{G}_{x}$, but first we
are going to introduce an auxiliary tool
$$\tilde{L}_{x} := \Set{s \in S}{x \in \dom{s}}.$$
We will introduce in $\tilde{L}_{x}$ an equivalence relation that
identifies two
elements $s,t \in \tilde{L}_{x}$ if, and only if, there is an idempotent
element $e \in \tilde{L}_{x}$ such that $se=te$.

The motivation for this process comes from the interpretation of the
well known concept of the isotropy group at a point in the unit space of
a groupoid, for the case of the groupoid of germs for an action of an inverse
semigroup on a space $X$, which we will explore later in this text.
Thus, the class of an element $s \in \till_x$ could be thought out as the
germ of $s$ at $x$, which also motivates the notation.

Summarizing, we have:
\begin{equation}\label{LHOrb0}
\begin{array}{rcl}
L_x & := & \Set{\germ{s}{x}}{x \in \dom{s}}, \\
G_x & := & \Set{\germ{s}{x}}{x \in \dom{s} \text{ and } \theta_s(x)=x}, \\
\orb(x) & := & \Set{\theta_s(x)}{x \in \dom{s}}.
\end{array}
\end{equation}

Notice that, if $s$ lies in $\till_{x}$ and $t$ lies in $\till_{\theta_s(x)}$, then
$ts$ lies in $\till_{x}$. Moreover, we have the following result.

\begin{proposition}\label{operation}
	Let $s \in \till_{x}$ and $t \in \till_{y}$, with $y = \theta_s(x)$.
	Then, $ts$ lies in $\till_{x}$
	and the class of $ts$ in $L_{x}$ depends only on
	the classes of $s$ and $t$ in $L_{x}$ and $L_{y}$, respectively.
\end{proposition}

\begin{proof}
	The first claim follows by the comment immediately before the proposition.
	For the second claim, notice initially that the fact that $\theta_s(x) = y$
	does not depend on representatives. Indeed, let $s' \in \till_{x}$ and
	$t' \in \till_{y}$ be elements such that
	$\germo{s'} = \germo{s}$ and $\germ{t'}{y} = \germ{t}{y}$.
	Therefore there are idempotents $e \in \till_{x}$ and
	$f \in \till_{y}$ such that $se = s'e$ and $tf = t'f$.
	We then have that
	$$ \theta_s(x) = \theta_s (\theta_e(x)) = \theta_{se}(x)
	= \theta_{s'e}(x) = \theta_{s'}(\theta_e(x)) = \theta_{s'}(x).$$
	It only remains to prove that $\germo{ts} = \germo{t's'}$. For this task, let
	$d$ be the idempotent given by $d = e s^{\ast} f s$ and notice that $d$
	lies in $\till_{x}$. Moreover,
	$$
	t' s' d
	= t' s' e s^{\ast} f s
	= t' s e s^{\ast} f s
	= t' s s^{\ast} f s e
	= t' f s s^{\ast} s e
	= t f s s^{\ast} s e
	= t s s^{\ast} f s e
	= t s e s^{\ast} f s
	= t s d,$$
	concluding the argument.
\end{proof}

Therefore, as long as $y=\theta_s(x)$, we are allowed to operate
the elements
$\germ{t}{y}$ and $\germo{s}$ to obtain $\germo{ts}$. This provides a group structure
on $G_{x}$ such that
$\germo{s}^{-1} = \germo{s^{\ast}}$ and whose identity element is $\germo{e}$ for any
idempotent element $e$ in $\till_{x}$. 

Notice that, whenever there is an idempotent element $e \in \tilde{L}_{x}$ such that $se=te$
for a pair of elements in $\tilde{L}_{x}$, necessarily $e$ lies in $\tilde{G}_{x}$,
since $\theta_e$ is the identity map on its domain. Therefore, $G_{x}$ coincides with
the maximal group image of $\tilde{G}_{x}$ and, from now on, will be called the
\emph{isotropy group} of $x$.

%

Notice that $L_xG_x \subseteq L_x$ and the map
$$\germo{s} \in L_x \mapsto \theta_s(x) \in \orb(x)$$
is well defined (by the proof of (\ref{operation})) and is onto.
Moreover, two elements $\germo{s}$ and $\germo{t}$ in $L_x$ satisfy
$\theta_s(x) = \theta_t(x)$, if and only if, $s^{\ast}t$ lies in $\tilh_x$.
Before we proceed, let us prove a technical result.

\begin{lemma}\label{hereditaryTill}
	Let $s,t \in S$ such that $s \leq t$. If $s$ lies in $\till_x$, then $t$
	must also lie in $\till_x$,
	$\germo{s}=\germo{t}$ and $\theta_s(x)=\theta_t(x)$. In particular,
	$\theta_t(\dom{s}) = \ran{s}$.
\end{lemma}

\begin{proof}
	Since $s \leq t$, there exists $e \in E(S)$ such that $s=te$. By hypothesis, $te=s \in \till_x$,
	which means that
	$$x \in \dom{(te)} = \dom{t} \cap X_e.$$
	Hence, $t$ and $e$ lie in $\till_x$. Moreover, $te=se$ and
	$\theta_s(x)=\theta_{te}(x)=\theta_t(\theta_e(x)) = \theta_t(x)$, as stated.
	
	Finally, if $x \in \dom{s}$, then $s \in \till_x$ and, by the previous part,
	$\theta_t(x)=\theta_s(x)$. This proves that
	$\theta_t(\dom{s}) = \theta_s(\dom{s}) = \ran{s}$.
\end{proof}

A central ingredient in the induction process is the vector space $M_{x}$
with basis $L_x$. We shall denote a basis element of $M_x$ by
$\delta_{\germo{s}}$ with $\germo{s}$ in $L_x$.
Since $L_xG_x \subseteq L_x$, $M_x$ has a natural right $KG_x$-module structure.

Consider the bilinear form
$$\iprod{\cdot}{\cdot}:M_x \times M_x \to KG_x$$
such that
$$
\iprod{ \delta_{\germo{s}}}{ \delta_{\germo{t}}}
=
\left\{
\begin{array}{c@{}l}
\delta_{\germo{s^{\ast}t}}, & \quad \text{if } s^{\ast}t \in \tilh_x, \\
\text{0}, & \quad \text{otherwise.}
\end{array}
\right.
$$

It is important to notice that this bilinear form is well
defined, that is, does not
depend on representatives. This said, we
shall also express this form as
$$
\iprod{\delta_{\germo{s}}}{\delta_{\germo{t}}}
=
\bool{ s^{\ast}t \in \tilh_x} \delta_{\germo{s^{\ast}t}},
$$
where the brackets indicate \textit{boolean value}
\footnote{
We shall often use boolean value, even in a slightly abusive way.
For example, in (\ref{LcsAction}), we have the expression
$\bool{st \in \till_x} f\big(\theta_{st}(x) \big) \delta_{\germo{st}}$.
Indeed, if $st$ does not lie in $\till_x$, then
it is not coherent to write $\theta_{st}(x)$. However, in
this case, we mean that the expression
equals $f\big(\theta_{st}(x) \big) \delta_{\germo{st}}$
if the content in the brackets is true and
the expression equals zero
otherwise.}.

An important property of this form, which may be easily proved, is expressed by the identity
\begin{equation}\label{khbalanced}
	\iprod{m}{na} = \iprod{m}{n}a,
\end{equation}
for all $m,n \in M_x$ and all $a \in KG_x$.

We say that $R_x \subseteq L_x$ is a
\emph{total system of representatives of left $G_x$-classes}
if, for every $\germo{s}$ in $L_x$, there exists precisely one element
$\germo{r}$ in $R_x$ such that $\theta_r(x)=\theta_s(x)$. By the previous
comment, this amounts to say that $r^{\ast}s$ lies in $\tilh_x$.

\begin{proposition}\label{totalset}
	If $R_x \subseteq L_x$ is a total system of representatives of left
	$G_x$-classes,
	then, for all $m \in M_x$, we have
	$$m = \sum_{\germo{r} \in R_x} \delta_{\germo{r}} \iprod{\delta_{\germo{r}}}{m},$$
	where the sum is always finite in the sense that there are only finitely many nonzero summands.
\end{proposition}

\begin{proof}
	Assume initially that $m=\delta_{\germo{s}}$ for some $\germo{s} \in L_x$.
	So, there exists an unique element $\germo{t}$ in $R_x$ such that
	$t^{\ast}s$ lies in $\tilh_x$.
	Hence,
	$$
	\sum_{\germo{r} \in R_x} \delta_{\germo{r}} \iprod{\delta_{\germo{r}}}{\delta_{\germo{s}}}
	= \sum_{\germo{r} \in R_x} \delta_{\germo{r}} \bool{ r^{\ast}s \in \tilh_x } \delta_{\germo{r^{\ast}s}}
	= \delta_{\germo{t}} \delta_{\germo{t^{\ast}s}}
	\overset{(\ref{hereditaryTill})}{=} \delta_{\germo{s}}.
	$$
	By writing $m$ as a combination of elements of the form $m=\delta_{\germo{s}}$
	for $\germo{s} \in L_x$, we may reach the general case.
\end{proof}

We can now derive a very important fact about
$M_x$. It is, there exists a left $\Lcs$-module structure compatible
with its right $KG_x$-module structure.

\begin{proposition}\label{LcsAction}
	There is a left $\Lcs$-module structure on $M_x$ such that
	$$
	(f \Delta_s) . \delta_{\germo{t}}
	= \bool{st \in \till_x} f\big(\theta_{st}(x) \big) \delta_{\germo{st}},
	$$
	for every $f \in \Lc{\ran{s}}$ and every $t \in \till_x$. Furthermore, with this structure $M_x$
	becomes a $\Lcs$-$KG_x$-bimodule.
\end{proposition}

\begin{proof}
	We shall prove first that there is a well defined left
	$\mathcal{L}(\mathcal{B}^{\theta})$-module structure on $M_x$,
	such that
	$$
	(f \Delta_s) . \delta_{\germo{t}}
	= \bool{st \in \till_x} f\big(\theta_{st}(x) \big) \delta_{\germo{st}}.
	$$
	Indeed, let $t$ and $t'$ in $\till_x$ such that $\germo{t}=\germo{t'}$. Hence,
	there exists $e \in \till_x$ such that $te=t'e$.
	
	Suppose $st \in \till_x$. Since $e \in \till_x$, we have:
	$$
	x \in \dom{(st)} \cap X_e
	= X_{e(st)^{\ast}(st)e}
	= X_{et^{\ast} s^{\ast} ste}
	= X_{e{t'}^{\ast} s^{\ast} st'e}
	= \dom{(st')} \cap X_e.
	$$
	Hence, $st'$ lies in $\till_x$. Moreover, $\germo{st}=\germo{st'}$, since
	$ste=st'e$ and
	$$
	\theta_{st}(x)
	= \theta_{st}(\theta_e(x))
	= \theta_{ste}(x)
	= \theta_{st'e}(x)
	= \theta_{st'}(\theta_e(x))
	= \theta_{st'}(x).
	$$
	Therefore, it is indeed a well defined action of
	$\mathcal{L}(\mathcal{B}^{\theta})$ on
	$M_x$. We shall show now that the ideal $\mathcal{J}_{\alpha}$,
	as in (\ref{StandingNotation}), acts trivially on
	$M_x$. For the task, let $\delta_{\germo{r}}$ in $M_x$ and
	let $b= f\Delta_s - f\Delta_t$ be an element in $\mathcal{J}_{\alpha}$,
	with $s \leq t$. First, notice that $sr \leq tr$, since $s \leq t$. 
	
	If $sr$ lies in $\till_x$, by (\ref{hereditaryTill}), we also have $tr \in \till_x$,
	$\germo{sr}=\germo{tr}$ and $\theta_{sr}(x) = \theta_{tr}(x)$. In this
	case, $b \cdot \delta_{\germo{r}}=0$. The same thing happens if both
	$sr$ and $tr$ do not lie in $\till_x$.
	
	Suppose now that $sr$ does not lie in $\till_x$, but $tr$ does. In this case,
	$\theta_{tr}(x)$ does not lie in $\ran{s}$, which contains the support of $f$.
	Indeed, otherwise,
	we should have $\theta_r(x) \in \theta_{t^{\ast}}(\ran{s})=\dom{s}$, where the last
	equality comes from (\ref{hereditaryTill}), since $s^{\ast} \leq t^{\ast}$.
	This gives $\theta_r(x) \in \dom{s} \cap \ran{r}$, from where we conclude
	that
	$$
	x \in \theta_{r^{\ast}}(\dom{s}\cap \ran{r})
	= X_{r^\ast s^\ast sr}
	= \dom{(sr)},
	$$
	which can not happen by assumption.
	
	Hence, the action factors trough $\mathcal{J}_{\alpha}$,
	giving an action of $\Lcs$
	on $M_x$,
	as desired. It is now standard to verify that $M_x$ is a $\Lcs$-$KG_x$-bimodule.
\end{proof}

We can now induce $\Lcs$-modules from $KG_x$-modules in the following way.
Given any left $KG_x$-module $V$, the tensor product
$$M_x \otimes_{KG_x} V,$$
is a left $\Lcs$-module, henceforth denoted simply by $M_x \otimes V$.

\begin{definition}\label{InducedModule}
	The $\Lcs$-module $M_x \otimes V$ mentioned above is said to be the module
	\emph{induced} by $V$ and, will be denoted by $\ind_{x}(V)$.
\end{definition}

The next lemma is a technical result which will be an important tool
to compute the annihilator
of the induced module in terms of
the annihilator of the original module $V$.

\begin{lemma}
	Let $V$ be a left $KG_x$-module and let $I$ be the annihilator of $V$ in $KG_x$.
	Given $m \in M_x$,
	the following are equivalent:
	\begin{enumerate}[(i)]
		\item $m \otimes v = 0$, for all $v \in V$;
		\item $\iprod{n}{m} \in I$, for all $n \in M_x$.
	\end{enumerate}
\end{lemma}

\begin{proof}
Let $n \in M_x$ and consider the bilinear map
$$(m,v) \in M_x \times V \mapsto \iprod{n}{m}v \in V.$$
By (\ref{khbalanced}), this is $KG_x$-balanced, so there is a well defined
$K$-linear map
$T_n: M_x \otimes V \to V$, such that
$$T_n(m \otimes v) = \iprod{n}{m}v,$$
So, if (i) is valid for $m \in M_x$, then
$$\iprod{n}{m}v = T_n(m \otimes v) = 0,$$
for all $n \in M_x$ and all $v \in V$.
Hence, $\iprod{n}{m}$ lies in the annihilator of $V$ for all $n \in M_x$, proving that
(ii) is valid for $m$ as well.

Conversely, let $m \in M_x$ and assume (ii) is valid for $m$. Let $R_x \subseteq L_x$ be a total system of representatives of left $G_x$-classes. Then, for every $v \in V$, we have
$$
m \otimes v \overset{(\ref{totalset})}{=}
\sum_{\germo{r} \in R_x} \delta_{\germo{r}} \iprod{\delta_{\germo{r}}}{m} \otimes v
= \sum_{\germo{r} \in R_x} \delta_{\germo{r}} \otimes \iprod{\delta_{\germo{r}}}{m} v = 0,
$$
proving that (i) is valid for $m$.
\end{proof}

We immediately obtain the following description for the annihilator of an induced module.

\begin{corollary}\label{NodependenceonV}
	Let $V$ be a left $KG_x$-module and let $I$ be the annihilator of $V$
	in $KG_x$. Then,
	$$
	\Set{b \in \Lcs}{\iprod{n}{bm} \in I, \forall n,m \in M_x},
	$$
	is the annihilator of $M_x \otimes V$ in $\Lcs$.
\end{corollary}

Notice that, if $I \trianglelefteq KG_x$, then $KG_x/I$ is a
left $KG_x$-module which is annihilated by $I$. Hence, every ideal
of $KG_x$ is the annihilator of a left $KG_x$-module. This motivates
the following definition.

\begin{definition}\label{InducedIdeal}
	Given any ideal $I \trianglelefteq KG_x$, we define
	$$\Ind{x}{I} := \Set{b \in \Lcs}{\iprod{n}{bm} \in I, \forall n,m \in M_x},$$
	and call it the \emph{ideal induced by} $I$.
\end{definition}

Notice that $\ind_x(I)$ is a two-sided ideal.
Moreover, notice that the annihilator of an induced
$\Lcs$-module
is the ideal induced from the annihilator of the original $KG_x$-module.
For further reference, we reinterpret (\ref{NodependenceonV}) from this
new point of view.

\begin{proposition}\label{AnnInducedVsIndIdeal}
	Let $V$ be a left $KG_x$-module and $I$ the annihilator of $V$ in $KG_x$. Then the
	annihilator of $M_x \otimes V$ coincides with the ideal induced by $I$.
\end{proposition}

We now start to explore the induction process by introducing
a clear fact about the behavior of
the induction process
under inclusion and intersection.

\begin{proposition}\hspace{1cm}
	\begin{enumerate}[(i)]
		\item If $I_1$ and $I_2$ are ideals of $KG_x$ with $I_1 \subseteq I_2$, then
		$\ind_x(I_1) \subseteq \ind_x(I_2)$.
		\item If $\left\{ I_{\lambda} \right\}_{\lambda \in \Lambda}$ is a family
		of ideals of $KG_x$,
		then
		$\ind_x \left(\bigcap_{\lambda \in \Lambda} I_\lambda \right)
		= \bigcap_{\lambda \in \Lambda} \ind_x(I_\lambda)$.
	\end{enumerate}
\end{proposition}

Notice that the task of checking that $\iprod{n}{bm} \in I$ for all $n, m\in M_x$,
as
required by the above definition, may be simplified by considering
$n = \delta_{\germo{s}}$ and
$m = \delta_{\germo{t}}$, for $s, t \in \till_x$, since these generate $M_x$.
So, the next result
is an important tool to use in this situation.

\begin{proposition}\label{ipcriteria}
	Given $b=\sum_{s \in S} f_s \Delta_s \in \Lcs$ and $k,l \in \till_x$,
	we have that
	$$
	\iprod{\delta_{\germo{k}}}{b \delta_{\germo{l}}}
	= \sum_{s \in K_x} f_s \big(\theta_k(x) \big) \delta_{\germo{k^{\ast}sl}},
	$$
	where $K_x$ is the set of all elements $s \in S$ such that $k^{\ast}sl$ lies in $\tilh_x$.
\end{proposition}

\begin{proof}
	By a simple computation, we obtain
	\begin{align*}
	\iprod{\delta_{\germo{k}}}{b \delta_{\germo{l}}}
	& = \sum_{s \in S} \iprod{\delta_{\germo{k}}}{ (f_s \Delta_s) \delta_{\germo{l}}}
	= \sum_{s \in S} \bool{sl \in \till_x} f_s \big(\theta_{sl}(x) \big)
			\iprod{\delta_{\germo{k}}}{\delta_{\germo{sl}} } \\
	& = \sum_{s \in K_x} \bool{sl \in \till_x} f_s \big(\theta_{sl}(x) \big) \delta_{\germo{k^{\ast}sl}}
		= \ldots
	\end{align*}
	Notice that, $k^{\ast}sl \in \tilh_x$ means that $x$ lies in the domains of
	$\theta_{k^{\ast}sl}$ and $\theta_{k^{\ast}sl}(x)=x$. By applying $\theta_k$
	in both sides of the last equality, we obtain $\theta_{sl}(x)=\theta_k(x)$. Nevertheless,
	if the right side of the last equality is well defined, so is the left side, which amounts to say
	that $sl$ lies in $\till_x$.
	Hence, the above equals
	$$
	\ldots
	= \sum_{s \in K_x} f_s \big(\theta_k(x) \big) \delta_{\germo{k^{\ast}sl}},
	$$
	as desired.
\end{proof}

By combining this proposition with the comment that motivated it immediately before, we get
a 
criteria for membership in $\ind_x(I)$:

\begin{proposition}\label{criteria}
	Given an ideal $I \trianglelefteq KG_x$ and
	$b = \sum_{s \in S} f_s \Delta_s \in \Lcs$, we have that
	$b \in \ind_x(I)$, if and only if,
	$$ \sum_{s \in K_x} f_s \big(\theta_k(x) \big) \delta_{\germo{k^{\ast}sl}} \in I,$$
	for every $k,l \in \till_x$.
\end{proposition}

We now proceed to introduce another fundamental concept in the induction process.
For each $x \in X$, consider the map $\Gamma_{x}: \Lcs \to KG_x$ given by
\begin{equation}\label{IntroduceGamma}
	\Gamma_{x} \left( \sum_{s \in S} f_s \Delta_s \right)
	= \sum_{s \in \tilh_x} f_s(x) \delta_{\germo{s}}.
\end{equation}

We shall show next that it is indeed a well defined map.

\begin{lemma}
	For every $x \in X$, the map $\Gamma_x$ introduced in
	(\ref{IntroduceGamma}) above is a well defined linear map.
\end{lemma}

\begin{proof}
	In fact, we shall show that the map $\Gamma'_x$, defined by
	$$
	\sum_{s \in S} f_s \bmdelta_s \in \mathcal{L}(\mathcal{B}^{\theta})
	\mapsto
	\sum_{s \in \tilh_x} f_s(x) \delta_{\germo{s}} \in KG_x
	$$
	vanishes on $\mathcal{J}_{\alpha}$.
	
	For the task, let $b=f \bmdelta_s -f \bmdelta_t$ lie in
	$\mathcal{J}_{\alpha}$, with $s \leq t$.
	Therefore, there are two possible scenarios for $s$:
	\begin{itemize}
		\item $s \in \tilh_x$: In this case, by (\ref{hereditaryTill}), $t$ also
		lies in $\tilh_x$ and $\germo{s}=\germo{t}$. Hence, $b$ lies in the kernel
		of $\Gamma'_x$.
		\item $s \notin \tilh_x$: In this case, again we have two distinct scenarios:
		
		If $s \in \till_x$, by (\ref{hereditaryTill}), $t \in \till_x$ and
		$\theta_t(x) =\theta_s(x) \neq x$. Hence, again $b$ lies in the kernel
		of $\Gamma'_x$.
		
		Otherwise, if both $s$ and $t$ does not lie in $\tilh_x$, $b$ lies in
		the kernel $\Gamma'_x$, and, if $s$ does not lie in $\tilh_x$ but
		$t$ does lie, we must have $s^\ast \notin \till_x$. Indeed, otherwise,
		since $s^\ast \leq t^\ast$, by (\ref{hereditaryTill}),
		$x = \theta_{t^{\ast}}(x) = \theta_{s^{\ast}}(x) \in \dom{s}$,
		contradicting the initial assumption. But this means that
		$x \notin \ran{s}$ which contains the support of $f$, and so
		$f(x)=0$. Hence, $b$ lies in the kernel of $\Gamma'_x$.
	\end{itemize}
	Therefore, $\Gamma'_x$ factors through $\mathcal{J}_{\alpha}$,
	giving the desired map $\Gamma_x$.
\end{proof}

The next lemma suggests a close relation between the maps $\Gamma_x$ and the
induction process.

\begin{lemma}\label{relationship}
	Let $k,l \in S$, $p \in \Lc{\dom{k}}$, $q \in \Lc{\ran{l}}$ and
	set
	$$u = p \Delta_{k^{\ast}} \text{ and } v=q \Delta_l.$$
	Then, for every $b \in \Lcs$, one has that
	$$\Gamma_x(ubv)
	=
	\left\{
	\begin{array}{c@{}l}
	p(x)q(\theta_l(x)) \iprod{\delta_{\germo{k}}}{b\delta_{\germo{l}}},
	& \quad \text{if } k,l \in \till_x, \\
	\text{0}, & \quad \text{otherwise.}
	\end{array}
	\right.
	$$
	\end{lemma}

\begin{proof}
	Writing $b=\sum_{s \in S} f_s \Delta_s$, we have
	$$
	ubv
	= \sum_{s \in S} p \Delta_{k^{\ast}} \cdot f_s \Delta_s \cdot q \Delta_l
	= \sum_{s \in S} p \balpha_{k^{\ast}}(f_s) \Delta_{k^{\ast}s} \cdot q \Delta_l
	= \sum_{s \in S} p \balpha_{k^{\ast}}(f_s) \balpha_{k^{\ast}s}(q) \Delta_{k^{\ast}sl}.
	$$
	Hence, by setting $K_x = \Set{s \in S}{k^{\ast}sl \in \tilh_x}$, we obtain
	$$
	\Gamma_x(ubv)
	=
	\sum_{s \in K_x} p(x) \bool{k \in \till_x} f_s(\theta_k(x))
				\bool{s^{\ast}k \in \till_x} q(\theta_{s^{\ast}k}(x)) \delta_{\germo{k^{\ast}sl}}
				= \ldots
	$$
	Notice that, if $s \in K_x$, then $\theta_{k^{\ast}sl}(x)=x$. By, applying $\theta_{s^{\ast}k}$
	in both sides of the last equality, we obtain
	$\theta_l(x) = \theta_{s^{\ast}k}(x)$.
	Hence, the above equals
	$$
	\ldots = \bool{k,l \in \till_x} p(x) q(\theta_l(x))
				\sum_{s \in K_x} f_s(\theta_k(x)) \delta_{\germo{k^{\ast}sl}}
	= \bool{k,l \in \till_x} p(x) q(\theta_l(x))
				\iprod{\delta_{\germo{k}}}{b \delta_{\germo{l}}},
	$$
	as desired.
\end{proof}

We now spell out an alternative definition of $\ind_x(I)$ in terms of $\Gamma_x$.

\begin{proposition}\label{IndGamma}
	If $I \trianglelefteq KG_x$ is an ideal, then
	$$\ind_x(I) = \Set{b \in \Lcs}{\Gamma_x(ubv) \in I \text{ for all } u,v \in \Lcs}.$$
\end{proposition}

\begin{proof}
Notice that, it is enough to prove that, for any $b \in \Lcs$, the following are equivalent:
\begin{enumerate}[(i)]
	\item $\Gamma_x(ubv) \in I$ for all $u,v \in \Lcs$;
	\item $\iprod{\delta_{\germo{k}}}{b \delta_{\germo{l}}} \in I$
	for all $k,l \in \till_x$.
\end{enumerate}
(i) $\Rightarrow$ (ii): Let $k,l \in \till_x$ and choose functions
$p \in \Lc{\dom{k}}$ and
$q \in \Lc{\ran{l}}$ such that $p(x)=1$ and $q(\theta_l(x))=1$. Letting
$u=p \Delta_{k^{\ast}}$ and $v= q \Delta_l$, by Lemma
(\ref{relationship}) we thus have
$$
I \ni \Gamma_x(ubv)
= p(x)q(\theta_l(x)) \iprod{\delta_{\germo{s}}}{b \delta_{\germo{l}}}
= \iprod{\delta_{\germo{s}}}{b \delta_{\germo{l}}}.
$$
(ii) $\Rightarrow$ (i): Conversely, it is enough to prove (i) for
$u=p \Delta_{k^{\ast}}$ and $v=q \Delta_l$, where $k$ and $l$ are
arbitrary elements in $S$.
By Lemma (\ref{relationship}), we have
$$
\Gamma_x(ubv)
= p(x) q(\theta_l(x)) \iprod{\delta_{\germo{k}}}{f \delta_{\germo{l}}} \in I,
$$
if $k$ and $l$ lie in $\till_x$, or
$$
\Gamma_x(ubv) = 0 \in I,
$$
otherwise, thus proving (i) in either case.
\end{proof}


\subsection{Admissible Ideals}

In this section we explore the relationship between induced ideals in $\Lcs$ and the ideals
of the isotropy group algebra they came from. In this context,
we introduce the concept of an
admissible ideal. Roughly speaking,
the admissible ideals are the ones which actually play a relevant role in the induction process.

For that task, again $\Gamma_x$ will play a relevant role and we begin by expelling out
an important behavior of $\Gamma_x$.

\begin{proposition}\label{quasihomo}
	Let $t \in \tilh_x$ and $\varphi \in \Lc{\ran{t}}$. Then, setting $a=\varphi \Delta_t$, we have that
	$$\Gamma_x(ab)=\Gamma_x(a)\Gamma_x(b)
	\text{ and } \Gamma_x(ba)=\Gamma_x(b)\Gamma_x(a),$$
	for every $b \in \Lcs$.
\end{proposition}

\begin{proof}
Let $b=\sum_{s \in S} f_s \Delta_s$ and notice that, $st$ lies in $\tilh_x$
if and only if $s$ lies in $\tilh_x$, since we already have $t \in \tilh_x$.
Then, by a simple computation, we get
\begin{align*}
	\Gamma_x(ba)
	& = \Gamma \left( \sum_{s \in S} f_s \balpha_s(\varphi) \Delta_{st} \right)
	= \sum_{s \in \tilh_x} f_s(x) \restr{\balpha_s(\varphi)}{x} \delta_{\germo{st}} \\
	& = \sum_{s \in \tilh_x} f_s(x) \varphi(x) \delta_{\germo{st}}
	= \left( \sum_{s \in \tilh_x} f_s(x) \delta_{\germo{s}} \right) \varphi(x) \delta_{\germo{t}}
	= \Gamma_x(b) \Gamma_x(a).
\end{align*}
Similarly, we can show $\Gamma_x(ab)=\Gamma_x(a)\Gamma_x(b)$, concluding the proof.
\end{proof}

\begin{proposition}\label{GammaIdeal}
	Let $J \trianglelefteq \Lcs$ be an ideal. Then, $\Gamma_x(J)$ is an
	ideal in $KG_x$. 
\end{proposition}

\begin{proof}
Let $a \in \Gamma_x(J)$ and $c=\delta_{\germo{t}} \in KG_x$ for some $t \in \tilh_x$.
Then, there exists $b \in J$
such that $\Gamma_x(b)=a$. Notice that, by choosing $\varphi \in \Lc{\ran{t}}$ such that
$\varphi(x)=1$, we have
$$
ac
= a \varphi(x) \delta_{\germo{t}}
= \Gamma_x(b) \Gamma_x(\varphi \Delta_t)
\overset{(\ref{quasihomo})}{=} \Gamma_x(b \cdot \varphi \Delta_t) \in \Gamma_x(J).
$$
By linearity, we deduce that $ac \in \Gamma_x(J)$ for arbitrary $c \in KG_x$ and
similarly, we can show that $ca \in \Gamma_x(J)$.
\end{proof}

We then have the following proposition.

\begin{proposition}\label{youaretheone}
	Let $I$ be an ideal in $KG_x$, and put $I'=\Gamma_x(\ind_x(I))$. Then
	\begin{enumerate}[(i)]
		\item $I'$ is an ideal of $KG_x$;
		\item $I' \subseteq I$;
		\item $\ind_x(I')=\ind_x(I)$.
	\end{enumerate}
\end{proposition}

\begin{proof}
	\begin{enumerate}[(i)]
		\item Follows from (\ref{GammaIdeal}).
		\item Given $a \in I'$, let $b \in \ind_x(I)$ such that $\Gamma_x(b)=a$.
		Notice that, if we choose an idempotent
		$e \in \tilh_x$ and $\varphi \in \Lc{X_e}$
		such that $\varphi(x)=1$ and then setting $d=\varphi \Delta_e$ we have
		$$
		a
		= \delta_{\germo{e}} a \delta_{\germo{e}}
		= \Gamma_x(d) \Gamma_x(b) \Gamma_x(d)
		\overset{(\ref{quasihomo})}{=} \Gamma_x(dbd)
		\in I,
		$$
		by (\ref{IndGamma}).
		\item The inclusion $\ind_x(I') \subseteq \ind_x(I)$ follows from (ii).
		On the other hand, if $b \in \ind_x(I)$, then $ubv \in \ind_x(I)$ for all
		$u,v \in \Lcs$.
		Hence,
		$$\Gamma_x(ubv) \in \Gamma_x(\ind_x(I)) = I'.$$
		By (\ref{IndGamma}), we conclude that $b \in \ind_x(I')$, as wanted.
	\end{enumerate}
\end{proof}

Notice that, if $I$ is an ideal of $KG_x$, then $I$ and
$\Gamma_x(\ind_x(I))$ induce the same ideal.
This motivates the following definition, which intends to identify the ideals that play
a relevant role in the induction process.

\begin{definition}\label{admissibleideal}
	An ideal $I \trianglelefteq KG_x$ is said to be \emph{admissible} if $\Gamma_x(\ind_x(I))=I$.
\end{definition}

In this setting, we have.

\begin{corollary}\label{AdmissibleInjectivity}
	For every ideal $I \trianglelefteq KG_x$, there exists an unique admissible ideal $I' \subseteq I$,
	such that $\ind_x(I')=\ind_x(I)$.
\end{corollary}

\begin{proof}
	Set $I'=\Gamma_x(\ind_x(I))$. By (\ref{youaretheone}.iii), we have $\ind_x(I')=\ind_x(I)$.
	Moreover,
	$$\Gamma_x(\ind_x(I')) = \Gamma_x(\ind_x(I)) = I',$$
	so $I'$ is admissible. Finally, if $I'$ and $I''$ are two admissible
	ideals inducing the same ideal
	of $\Lcs$, then
	$$I' = \Gamma_x(\ind_x(I')) = \Gamma_x(\ind_x(I''))=I''.$$
\end{proof}

We already have two examples of induced ideals.

\begin{proposition}
	The trivial ideals of $KG_x$ are admissible.
\end{proposition}

\begin{proof}
	Notice that
	$$
	\left\{ 0 \right\}
	\subseteq \Gamma_x(\ind_x( \left\{ 0 \right\})) 
	\overset{(\ref{youaretheone}.ii)}{\subseteq}
	\left\{ 0 \right\}
	$$
	so $\left\{ 0 \right\}$ is admissible.
	On the other hand, we have
	$$\Gamma_x(\ind_x(KG_x)) = \Gamma_x(\Lcs) = KG_x,$$
	so $KG_x$ is admissible.
\end{proof}

Once we have studied the relationship of $I$ and $\Gamma_x(\ind_x(I))$,
in the case $I$ is
an ideal of $KG_x$, we may ask ourselves about the relationship of $J$ and $\ind_x(\Gamma_x(J))$,
in the case $J$ is an ideal in $\Lcs$.

\begin{proposition}\label{largestideal}\hspace{1cm}
	\begin{enumerate}[(i)]
		\item If $J \trianglelefteq \Lcs$, then $J \subseteq \ind_x(\Gamma_x(J))$.
		\item If $I \trianglelefteq KG_x$, then $\ind_x(I)$ is the largest
		among the ideals
		$J \trianglelefteq \Lcs$ satisfying $\Gamma_x(J) \subseteq I$.
	\end{enumerate}
\end{proposition}

\begin{proof}\hspace{1cm}
	\begin{enumerate}[(i)]
		\item If $b \in J$, then for every $u,v \in \Lcs$, $ubv \in J$ and so
		$$\Gamma_x(ubv) \in \Gamma_x(J),$$
		from where we deduce, by (\ref{IndGamma}), that $b \in \ind_x(\Gamma_x(J))$.
		\item By (\ref{youaretheone}.ii), $\Gamma_x(\ind_x(I)) \subseteq I$. Moreover,
		if $J \trianglelefteq \Lcs$ satisfies $\Gamma_x(J) \subseteq I$, then
		$$J \subseteq \ind_x(\Gamma_x(J)) \subseteq \ind_x(I).$$
	\end{enumerate}
\end{proof}

Finally, $\Gamma_x$ always leads to admissible ideals.

\begin{proposition}\label{GammaAdmissible}
	For any ideal $J \trianglelefteq \Lcs$, $\Gamma_x(J)$ is an admissible
	ideal of $KG_x$.
\end{proposition}

\begin{proof}
	By (\ref{youaretheone}.i), $J \subseteq \ind_x(\Gamma_x(J))$. Hence,
	$$
	\Gamma_x(J)
	\subseteq \Gamma_x(\ind_x(\Gamma_x(J)))
	\overset{(\ref{youaretheone}.ii)}{\subseteq} \Gamma_x(J).
	$$
	Therefore, $\Gamma_x(J)$ is admissible.
\end{proof}

%% file: representation.tex
\section{Effros-Hahn like Theorem}

In this section we shall prove that any ideal
(always meaning two-sided ideal)
of $\Lcs$ is the intersection of ideals induced from isotropy subgroups.
Our methods will largely rely on representation theory, as we shall see.

\subsection{Representations}
\label{representation}

\textbf{From this point on, we will fix an arbitrary ideal
$J \trianglelefteq \Lcs$ which, in view of
(\ref{nondegrep}) and (\ref{localunits}), we may assume it is
the kernel of a likewise fixed non-degenerate representation
$$\pi: \Lcs \to L(V).$$}

By (\ref{disintegration0}), we may disintegrate it to
a covariant representation $(\pi_0,\sigma)$ of
$(\theta, S, X)$ such that
$$\pi(f \Delta_s) = \pi_0 (f) \sigma_s.$$

By the comment immediately after Proposition (\ref{LcImmerse}),
we may identify $\Lc{X}$ as a subalgebra of $\Lcs$ and, in this fashion,
we can also interpret $\pi_0$ as the restriction of $\pi$ to $\Lc{X}$.
Hence, by an abuse of notation,
we shall also use $\pi$ to denote $\pi_0$ from now on.
The context should be enough to
distinguish between the initial representation $\pi$ of $\Lcs$
and the representation $\pi$ of $\Lc{X}$ composing the covariant
representation $(\pi,\sigma)$ resulted from the disintegration of $\pi$. 

Notice that the definition of induced ideals requires that a point of $X$ to be chosen
in advance, so we must begin to see our representation $\pi$ from the point of view of a chosen
point in $X$, a process which will eventually lead to a \emph{discretization} of $\pi$.

For each $x \in X$, let
$$I_x=\Set{f \in \Lc{X}}{f(x)=0}$$
which is clearly an ideal in $\Lc{X}$. Consequently,
$$Z_x := \Span\left\{ \pi(I_x)V \right\}$$
is invariant under $\Lc{X}$, so there is a well defined representation $\pi_x$ of $\Lc{X}$ on
$V_x := V/Z_x$ making the following diagram
\[
\begin{tikzcd}
V \arrow{r}{\pi(f)} \arrow[swap]{d}{q_x} & V \arrow{d}{q_x} \\
V_x \arrow[swap]{r}{\pi_x(f)} & V_x
\end{tikzcd}
\]
to commute for every $f \in \Lc{X}$.

The next proposition is an indication that the localization process is bearing fruits.

\begin{proposition}\label{eigenvalue}
	Let $x \in X$ and $f \in \Lc{X}$. Then, for every $\eta \in V_x$, we have
	$$\pi_x(f) \eta = f(x) \eta.$$
\end{proposition}

\begin{proof}
	Since $\pi$ is non-degenerate, it is enough to verify for $\eta = q_x(\pi(\varphi)\xi)$.
	Let $C$ be a compact open set containing $\supp(\varphi) \cup \left\{ x \right\}$ and notice that
	$1_C \varphi = \varphi$. Furthermore, $f-f(x)1_C$ lies in $I_x$ and hence
	$$
	f \varphi
	= \big(f - f(x)1_C + f(x)1_C \big)\varphi
	= (f - f(x)1_C)\varphi + f(x)\varphi
	\overset{\Mod I_x}{\equiv} f(x)\varphi.
	$$
	We thus obtain
	$$ \pi(f\varphi) \xi
	\overset{\Mod Z_x}{\equiv} \pi(f(x)\varphi) \xi
	= f(x) \pi(\varphi) \xi.
	$$
	Therefore,
	\begin{align*}
	\pi_x(f)\eta & = \pi_x(f)q_x(\pi(\varphi)\xi) = q_x(\pi(f)\pi(\varphi) \xi)
	= q_x(\pi(f \varphi)\xi) \\
	& = q_x(f(x)\pi(\varphi)\xi) = f(x)q_x(\pi(\varphi)\xi) = f(x) \eta,
	\end{align*}
	concluding the proof.
\end{proof}

Combining the definition of $\pi_x$ with the result above, we get the following useful formula
\begin{equation}\label{joineigenvalue}
	q_x(\pi(f)\xi) = \pi_x(f)q_x(\xi) = f(x) q_x(\xi)
\end{equation}
for all $x \in X$, $f \in \Lc{X}$ and $\xi \in V$.

Now, we shall work with the homomorphism $\sigma$ and, for a moment,
let the maps $\pi_x$ for aside.

\begin{proposition}
	If $x \in \dom{s}$, then $\sigma_s(Z_x) \subseteq Z_{\theta_s(x)}$.
	Furthermore, there exists a linear mapping
	$$\mu_s^x: V_x \to V_{\theta_s(x)},$$
	such that, for all $\xi \in V$
	$$\mu_s^x(q_x(\xi)) = q_{\theta_s(x)}(\sigma_s(\xi)).$$
\end{proposition}

\begin{proof}
	Let $\xi \in Z_x$ be a vector of the form $\xi = \pi(\varphi) \eta$,
	where $\varphi \in I_x$ and $\eta \in V$. In such case,
	$$
	\sigma_s(\xi)
	= \sigma_s \pi(\varphi) \eta
	\overset{(\ref{helpful2})}{=} \pi(\bar{\alpha}_s(\varphi)) \sigma_s \eta.
	$$
	Noticing that
	$$
	\restr{\balpha_s(\varphi)}{\theta_s(x)}
	= \varphi(\theta_{s^{\ast}}(\theta_s(x)))
	= \varphi(x)=0,
	$$
	we see that $\balpha_s(\varphi)$ lies in $I_{\theta_s(x)}$. Hence,
	$\sigma_s(\xi)$ lies in $Z_{\theta_s(x)}$. The result then follows
	by linearity and the second part is an immediate consequence of the first.
\end{proof}

We shall explore some properties of the maps $\mu_s^x$.

\begin{proposition}\label{mueident}
	Let $e \in E(S)$ and $x \in X_e$, then $\mu_e^x$ is the identity map in $V_x$.
\end{proposition}

\begin{proof}
Choose $\varphi \in \Lc{X_e}$ such
that $\varphi(x)=1$. Then,
\begin{equation}\label{GonnaBeAlright}
	q_x(\pi(\varphi)\eta)
	\overset{(\ref{joineigenvalue})}{=} \varphi(x) q_x(\eta)
	= q_x(\eta)
\end{equation}
for all $\eta \in V$.

Let $\xi \in V$, then using (\ref{GonnaBeAlright})
for $\eta = \sigma_e(\xi)$, we have
\begin{equation*}
	\mu_e^x(q_x(\xi))
	= q_x(\sigma_e(\xi))
	\overset{(\ref{GonnaBeAlright})}{=} q_x(\pi(\varphi) \sigma_e \xi)
	\overset{(\ref{helpful})}{=} q_x(\pi(\varphi)\xi)
	\overset{(\ref{joineigenvalue})}{=} \varphi(x) q_x(\xi)
	= q_x(\xi),
\end{equation*}
concluding the proof.
\end{proof}

The maps $\mu_s^x$ obey the following functorial property.

\begin{proposition}\label{mufunctoriality}
	If $x \in X_{s^{\ast}t^{\ast}ts}$, then the composition
	\[
	\begin{tikzcd}
	V_x \arrow{r}{\mu_s^x} & V_{\theta_s(x)} \arrow{r}{\mu_t^{\theta_s(x)}} & V_{\theta_{ts}(x)} \\
	\end{tikzcd}
	\]
	coincides with $\mu_{ts}^x$.
\end{proposition}

\begin{proof}
	First notice that, since $x \in X_{s^{\ast}t^{\ast}ts}$, we must
	have $x \in \dom{s}$
	and $\theta_s(x) \in \dom{t}$. Hence, for $\xi \in V$, we have
	\begin{align*}
	\mu_t^{\theta_s(x)} \Big( \mu_s^x \big( q_x(\xi) \big) \Big)
	& = \mu_t^{\theta_s(x)} \Big( q_{\theta_s(x)} \big( \sigma_s(\xi) \big) \Big)
	= q_{\theta_t (\theta_s(x))} \Big( \sigma_t \sigma_s(\xi) \Big) \\
	& = q_{\theta_{ts}(x)} \big( \sigma_{ts}(\xi) \big)
	=\mu_{ts}^x \big( q_x(\xi) \big),
	\end{align*}
	proving the statement.
\end{proof}

Let us now consider the representation of $\Lc{X}$ on the cartesian product
$\Pi_{x \in X} V_x$ given by
$$\Pi = \prod_{x \in X} \pi_x.$$
Thus, if $f \in \Lc{X}$, and
$\eta=(\eta_x)_{x \in X} \in \prod_{x \in X} V_x$, we have
$$(\Pi(f)\eta)_x = \pi_x(f)\eta_x \overset{(\ref{eigenvalue})}{=} f(x)\eta_x$$
for all $x \in X$.

Hence, $\Pi(f)$ is the block diagonal operator, acting on each $V_x$ as
scalar multiplication by $f(x)$.

Also, for each $s \in S$, consider the linear operator $U_s$ on $\prod_{x \in X} V_x$, given by
$$U_s(\eta)_x = \bool{x \in \ran{s}} \mu_s^{\theta_{s^{\ast}}(x)} (\eta_{\theta_{s^{\ast}}(x)})$$
for all $\eta=(\eta_x)_{x \in X} \in \prod_{x \in X} V_x$.

\begin{proposition}\label{Ufunctoriality}
	Identifying $V_x$ as a subspace of
	$\prod_{x \in X}V_x$, in the natural way, we have:
	\begin{enumerate}[(i)]
		\item if $x \notin \dom{s}$, then $U_s$ vanishes on $V_x$;
		\item if $x \in \dom{s}$, then $U_s$ coincides with $\mu_s^x$ and hence maps
		$V_x$ to $V_{\theta_s(x)}$;
		\item if $x \in \dom{s}$, then $U_s$ maps
		$V_x$ bijectively onto $V_{\theta_s(x)}$;
		\item if $x \in X_{s^{\ast}t^{\ast}ts}$, then the composition
		\[ \begin{tikzcd}
		V_x \arrow{r}{U_s} & V_{\theta_s(x)} \arrow{r}{U_t} & V_{\theta_{ts}(x)} \\
		\end{tikzcd}
		\]
		coincides with $U_{ts}$ on $V_x$;
		\item $U$ is a semigroup homomorphism.
	\end{enumerate}
\end{proposition}

\begin{proof}
	Items (i) and (ii) are easy to see, while item (iv) follows immediately
	by (\ref{mufunctoriality}).
	For item (iii), is is enough to notice that, by (iv),
	the restriction of
	$U_{s^{\ast}}$ to $V_{\theta_s(x)}$ is the inverse of $U_s$
	restricted to $V_x$. Finally, in order to prove (v), notice that,
	$U_t \circ U_s = U_{ts}$ on $V_x$ for every $x \in X_{s^{\ast}t^{\ast}ts}$
	by item (iv),
	and, for $ x \notin X_{s^{\ast}t^{\ast}ts}$, we must have
	$x \in \dom{s}$ or $\theta_s(x) \notin \dom{t}$. In either case,
	$U_t \circ U_s$ vanishes in $V_x$ as $U_{ts}$ does by item (i).
\end{proof}

Unfortunately, since we are considering the product of the $V_x$
instead of the direct sum, the pair $(\Pi, U)$ may not be a covariant
representation of the system $(\theta, S, X)$. However,
we can get around this problem with bare hands. The reason why we insist
in maintain the product of the $V_x$ will be clear later. 

\begin{proposition}
	The pair $(\Pi, U)$ can be integrated to a representation $\Pi \times U$
	of $\Lcs$ on $\prod_{x \in X} V_x$ such that
	$(\Pi \times U)(f \Delta_s) = \Pi(f) U_s$.
\end{proposition}

\begin{proof}
	We already know that $U$ is a semigroup homomorphism and
	$\Pi$ is a representation of $\Lc{X}$ on $\prod_{x \in X} V_x$,
	possibly degenerate.
	Furthermore, the pair $(\Pi, U)$ may not satisfy condition
	(ii) of (\ref{covrep}). However, it still satisfies condition (i).
	Indeed, for $s \in S$, $f \in \Lc{\dom{s}}$ and
	$\eta \in \prod_{x \in X} V_x$, we have
	for all $x \in X$
	\begin{align*}
		U_s \Pi(f) U_{s^{\ast}}(\eta)_x
		& = \bool{x \in \ran{s}} \mu_s^{\theta_{s^{\ast}}(x)}
			\bigg(\Pi(f)U_{s^{\ast}}(\eta)_{\theta_{s^{\ast}}(x)} \bigg)
		= \bool{x \in \ran{s}} \mu_s^{\theta_{s^{\ast}}(x)} 
		 	\bigg( f\big( \theta_{s^{\ast}}(x) \big)
				U_{s^{\ast}}(\eta)_{\theta_{s^{\ast}}(x)} \bigg) \\
		& = \bool{x \in \ran{s}} f\big( \theta_{s^{\ast}}(x) \big)
			 \mu_s^{\theta_{s^{\ast}}(x)}
			\mu_{s^{\ast}}^{\theta_s(x)} (\eta_x)
			\overset{(\ref{mufunctoriality})}{=}
				\bool{x \in \ran{s}} \restr{\alpha_s(f)}{x} \eta_x \\
		& = \restr{\alpha_s(f)}{x} \eta_x
		= \Big( \Pi(\alpha_s(f))\eta \Big)_x.
	\end{align*}
	Now, let $e \in E(s)$, $f \in \Lc{X_e}$ and $\xi \in \prod_{x \in X} V_x$.
	Then, let $\eta$ be the vector in $\prod_{x \in X} V_x$ such that
	$\eta_x = f(x) \xi_x$ and notice that
	$$
	U_e(\eta)_x
	= \bool{x \in X_e} \mu_e^x \Big( f(x) \xi_x \Big)
	= \bool{x \in X_e} f(x) \xi_x
	= f(x) \xi_x
	= \Big( \Pi(f) \xi \Big)_x.
	$$
	Hence, $\eta$ is a vector such that $U_e(\eta) = \Pi(f) \xi$.
	Then, we can replace the use of condition (ii) of Definition
	(\ref{covrep}) by this argument in the proof of Lemma
	(\ref{helpful}) to obtain
	$$U_e \Pi(f) = \Pi(f) = \Pi(f) U_e.$$
	Replacing now the use of Lemma (\ref{helpful}) by the equality above
	in the proof of
	(\ref{integration})
	we obtain a representation of the semi-direct product bundle
	$\mathcal{B}^{\theta}$, possibly degenerate, which can be further
	integrated to a representation $\Pi \times U$ of $\Lcs$ on
	$\prod_{x \in X} V_x$ such that $(\Pi \times U)(f \Delta_s) = \Pi(f) U_s$.
	\end{proof}

\begin{definition}
	The representation $\discr$ above will be referred as the $\emph{discretization}$ of
	the initially given representation $\pi$.
\end{definition}

The reader may ask himself or herself why we have not
considered the discretized
representation
acting on the direct sum of the $V_x$, instead of their product.
The map which will be introduced in the next proposition
is the main reason for that
since it is an important tool to establish a relation
between the null space of the original representation and
the null space
of its
discretized form as we shall see.
%

\begin{proposition}\label{Qcovariant}
	The mapping
	$$Q: \xi \in V \mapsto (q_x(\xi))_{x \in X} \in \prod_{x \in X} V_x,$$
	is injective and equivariant
	\footnote{Recall that a linear map $T: E \to F$ between vector spaces
	$E$ and $F$ is equivariant relative to representations
	$\pi_E$ and $\pi_F$ of an algebra $A$ on $E$ and $F$, respectively,
	if it intertwines $\pi_E$ and $\pi_F$, meaning that
	$T \circ \pi_E(a) = \pi_F(a) \circ T$ for every $a$ in $A$.}
	relative to the corresponding representations of $\Lcs$ on $V$
	and on $\prod_{x \in X} V_x$, respectively.
\end{proposition}

\begin{proof}
	Let $s \in S$, and $f \in \Lc{\ran{s}}$. Then, for every $\xi $ in $V$, and every
	$x \in X$, we have
	\begin{align*}
	\big( (\Pi \times U)(f \Delta_s) Q(\xi) \big)_x
	& = \big( \Pi(f) U_s Q(\xi) \big)_x
	= f(x) \big( U_s Q(\xi) \big)_x \\
	& = f(x) \bool{x \in \ran{s}} \mu_s^{\theta_{s^{\ast}(x)}}
				\big( Q(\xi)_{\theta_{s^{\ast}(x)}} \big)
	= f(x) \mu_s^{\theta_{s^{\ast}(x)}} \big( q_{\theta_{s^{\ast}}(x)}(\xi) \big) \\
	& = f(x) q_x(\sigma_s(\xi) \big)
	= q_x \big( \pi(f) \sigma_s(\xi) \big)
	= Q \big( \pi(f\Delta_s)\xi \big)_x.
	\end{align*}
	This proves that $Q$ is covariant.	
	In order to prove that $Q$ is injective, suppose that $Q(\xi)=0$, for some
	$\xi $ in $V$. Since $\pi$ is non-degenerate, there exists $f_i \in \Lc{X}$ and $\xi_i \in V$
	such that
	$$
	\xi = \sum_{i=1}^n \pi(f_i) \xi_i,
	$$
	Define
	$$
	D = \bigcup_{i=1}^n \supp(f_i),
	$$
	so $D$ is a compact open subset of $X$ and we have
	\begin{equation}\label{XiSuppD}
	\xi
	= \sum_{i=1}^n \pi(1_D f_i) \xi_i
	= \pi (1_D) \sum_{i=1}^n \pi(f_i)\xi_i
	= \pi(1_D) \xi.
	\end{equation}
	
	Now, for each $x$ in $X$, we have $q_x(\xi)=0$ by hypothesis. Thus,
	$\xi$ lies in $Z_x$ and so we may write
	$$
	\xi = \sum_{i=1}^{n_x} \pi(f^{(x)}_i) \xi^{(x)}_i,
	$$
	where $\xi^{(x)}_i \in V$ and $f^{(x)}_i \in I_x$. Since
	there are finitely many
	$f^{(x)}_i$, each of which locally constant, there exists a
	compact open neighborhood $C_x$ of $x$ where all of the $f^{(x)}_i$ vanish.
	Moreover,
	\begin{equation}\label{LocalValueZero}
		\pi(1_{C_x})\xi
		= \sum_{i=1}^{n_x} \pi(1_{C_x} f^{(x)}_i)\xi^{(x)}_i
		= 0.
	\end{equation}
	
	Finally, $\left\{C_x \right\}_{x \in X}$ is an open cover of $D$, and hence
	we may find a finite set
	$\left\{x_1,\ldots,x_p \right\} \subseteq X$, such that $D \subseteq \bigcup_{i=1}^p C_{x_i}$.  Putting
	$$
	E_k = D \cap C_{x_k} \setminus \bigcup_{i=1}^{k-1} C_{x_i},
	$$
	for $k=1,\ldots ,p$, it is easy to see that the $E_k$ are pairwise disjoint
	compact open sets, whose
	union coincides with $D$. Observing that $E_k \subseteq C_{x_k}$, we then have
	$$
	\xi
	\overset{(\ref{XiSuppD})}{=} \pi(1_D)\xi
	= \sum_{k=1}^p \pi(1_{E_k}) \xi
	= \sum_{k=1}^p \pi(1_{E_k} 1_{C_{x_k}})\xi
	= \sum_{k=1}^p \pi(1_{E_k}) \pi(1_{C_{x_k}})\xi
	\overset{(\ref{LocalValueZero})}{=} 0.
	$$
	This proves that $Q$ is injective.
\end{proof}

Thus, we have an immediate consequence.

\begin{corollary}\label{FirstNullSpace}
	The null space of $\discr$ is contained in the null space of $\pi$.
\end{corollary}

\begin{proof}
	Let $b \in \ker(\discr)$. By (\ref{Qcovariant}), we have
	$$0 = (\discr)(b) Q(\xi) = Q \big( \pi(b)\xi \big)$$
	for all $\xi \in V$. Again by (\ref{Qcovariant}), $Q$ is injective and, hence,
	$$\pi(b)\xi = 0$$
	for all $\xi \in V$, that is, $b \in \ker(\pi)$.
\end{proof}

From now on, we shall consider the subspace
$$\bigoplus_{x \in X} V_x \subseteq \prod_{x \in X} V_x,$$
consisting of the vectors with finitely many nonzero coordinates. It is easy to see that this subspace
is invariant under $\Pi(f)$ for all $f \in \Lc{X}$, as well as under $U_s$ for all $s \in S$.
Consequently, it is also invariant under $\discr$.

\begin{proposition}\label{SecondNullSpace}
	The null space of the representation obtained by restricting $\discr$ to
	$\bigoplus_{x \in X} V_x$ coincides with the null space of $\discr$ itself.
\end{proposition}

\begin{proof}
	Suppose that $(\discr)(b)$ vanishes on $\bigoplus_{x \in X} V_x$
	for some $b = \sum_{s \in S} f_s \Delta_s \in \Lcs$.
	Let $y \in X$ and $\eta = (\eta_x)_{x \in X} \in \prod_{x \in X} V_x$
	and notice that
	$$
	((\discr)(b)\eta)_y
	= \sum_{s \in S} (\Pi(f_s)U_s \eta)_y
	= \sum_{s \in S} f_s(y) \bool{y \in X_{ss^{\ast}}} \mu^{\theta_{s^{\ast}}(y)}_s(\eta_{\theta_{s^{\ast}}(y)}).
	$$
	Let $\eta'=(\eta'_x)_{x \in X}$ be the vector defined by either $\eta'_x = \eta_x$ if $x=\theta_{s^{\ast}}(y)$
	for some $s \in S$ such that $y \in \ran{s}$ and $f_s \neq 0$, or $\eta'_x=0$ otherwise. Then,
	it is clear that $\eta' \in \bigoplus_{x \in X} V_x$ and
	$$((\discr)(b)\eta)_y = ((\discr)(b)\eta')_y = 0.$$
	Since $y \in X$ and $\eta \in \prod_{x \in X} V_x$ are arbitrary, we deduce that $(\discr)(b)=0$, concluding the argument.
\end{proof}

Regarding the space $\bigoplus_{x \in X} V_x$ where $\discr$ acts,
we will identify each $V_x$ as a subspace of  $\bigoplus_{x \in X} V_x$, in the usual way.
Thus, given $\xi \in V$, we shall think of $q_x(\xi )$ as the element
of $\bigoplus_{x \in X} V_x$ whose coordinates all vanish,
except for the $x^{\text{th}}$ coordinate which takes on the value $q_x(\xi)$.
In this fashion, notice that
\begin{equation}\label{computeeasy}
	\begin{array}{rcl}
	\Pi(f) q_x(\xi ) = & \pi_x(f)q_x(\xi ) & = q_x \big( \pi(f) \xi \big ), \\
	U_s \big( q_x(\xi ) \big ) = & \bool{x \in \dom{s}} \mu_s^x \big( q_x( \xi )\big ) & = \bool{x \in \dom{s}}
	q_{\theta_s(x)}(\sigma_s \xi),
	\end{array}
\end{equation}
for all $f\in \Lc{X}$, $s \in S$, $x \in X$, and $\xi \in V$.

Since $\bigoplus_{x \in X} V_x$ is spanned by the union of the
$V_x$, each of which is the range of the corresponding $q_x$,
the formulas above determine the action of $\Pi(f)$ and
$U_s$ on the whole space $\bigoplus_{x \in X} V_x$. So, by
combining them, we are able to give the following concrete
description
of the restriction of $\discr$ to $\bigoplus_{x \in X} V_x$.

\begin{proposition}\label{ExpressionForPixU}
	Let $b = \sum_{s \in S} f_s \Delta_s$ in $\Lcs$. Then, for
	all $x \in X$
	and $\xi \in V$, we have that
	$$
	(\discr)(b)q_x(\xi )
	= \sum_{s \in S} \bool{x \in \dom{s}}
			q_{\theta_s(x)} \big( \pi(f_s)\sigma_s \xi \big).
	$$
\end{proposition} 

\begin{proof}
	By (\ref{computeeasy}), the proof reduces to a direct computation:
	\begin{align*}
	(\discr)(b)q_x(\xi )
	& = \sum_{s \in S} \Pi(f_s)U_s \big( q_x(\xi ) \big)
	= \sum_{s \in S} \Pi(f_s) \bool{ x \in \dom{s} }
				q_{\theta_s(x)} (\sigma_s \xi) \\
	& = \sum_{s \in S} \bool{x \in \dom{s}}
				q_{\theta_s(x)} \big( \pi(f_s) \sigma_s \xi \big).
	\end{align*}
\end{proof}

We are going to describe now the \emph{matrix entries} of the
operator $(\discr)(b)$ acting on $\bigoplus_{x \in X} V_x$.
That is, for each $x$ and $y$ in $X$, we want an expression for
the $y^{\text{th}}$ component of the vector obtained by applying
$(\discr)(b)$ to any given vector in $V_x$, say of the form
$q_x(\xi)$, where $\xi \in V$.

It is clear that the desired expression is the $y^{\text{th}}$
component of the expression
given in (\ref{ExpressionForPixU}),
which is in turn given by the partial sum corresponding to
the terms for which $\theta_s(x)=y$. So, we have
\begin{equation}\label{MatrixEntries}
	\big( (\discr)(b) q_x(\xi ) \big)_y
	= \sum_{s \in S, \theta_s(x)=y} q_{\theta_s(x)} \big( \pi(f_s) \sigma_s \xi \big)
	= q_y \bigg(\sum_{s \in S, \theta_s(x)=y} \pi(f_s) \sigma_s \xi \bigg).
\end{equation}

We are going to prove now that the restriction of the discretized representation $\discr$
to $\bigoplus_{x \in X} V_x$ has the same null space as the original representation
$\pi$ has. But first, recall that in (\ref{FirstNullSpace}) and (\ref{SecondNullSpace})
we already proved the following relations among the null spaces:
\begin{equation}\label{SoFarKernelInclusions}
	\ker( \pi ) \supseteq \ker(\discr) = \ker \big( \discr |_{\oplus_{x\in X}V_x} \big).	
\end{equation}

We are going to show now that equality in fact holds throughout.

\begin{theorem}\label{ThirdNullSpace}
	The null space of the representation obtained by restricting $\discr$
	to $\bigoplus_{x \in X} V_x$ coincides with the null space of $\pi$.	
\end{theorem}

\begin{proof}
	It is important to emphasize that,
	since $(\Pi \times U)(b)$ is well defined on each $V_x$,
	then so is the right-hand-side in (\ref{MatrixEntries}).
	Precisely, if $\xi$ and $\xi'$ are elements of $V$ such
	that $q_x(\xi) = q_x(\xi')$, then
	\begin{equation}\label{WellDef}
		q_y \bigg( \sum_{s \in S, \theta_s(x)=y}
				\pi(f_s) \sigma_s \xi \bigg)
		= q_y \bigg( \sum_{s \in S, \theta_s(x)=y}
				\pi (f_s) \sigma_s \xi' \bigg).
	\end{equation}
	
	By (\ref{SoFarKernelInclusions}),
	if $b$ is in the null space of $\pi$, it is enough to prove that
	$(\Pi \times U)(b)$ vanishes on $\bigoplus_{x \in X} V_x$, which
	amounts to prove that its matrix
	entries given by (\ref{MatrixEntries}) vanish for
	all $x$ and $y$ in $X$.
	
	Let $b =\sum_{s \in S} f_s \Delta_s$ and $\Lambda \subseteq S$
	be the subset consisting of those $s$ for which $f_s \neq 0$,
	and notice that $\Lambda$
	decomposes as the disjoint union of the following subsets:
	\begin{align*}
	& \Lambda_1 = \Set{s \in \Lambda}{y \notin \ran{s}}, \\
	& \Lambda_2 = \Set{s \in \Lambda}{y \in \ran{s} ,\ \theta_{s^{\ast}}(y) \neq x}, \\
	& \Lambda_3 = \Set{s \in \Lambda}{y \in \ran{s} ,\ \theta_{s^{\ast}}(y)=x}.
	\end{align*}

	From our hypothesis that $\pi(b)=0$, we have that, for every $\eta$ in $V$,
	\begin{equation}\label{HypothesisPiVanishes}
		0 = \pi(b) \eta
		= \sum_{s \in \Lambda} \pi(f_s\Delta_s)\eta
		= \sum_{s \in \Lambda} \pi(f_s) \sigma_s \eta.
	\end{equation}
	
	Comparing this expression with the last part of
	(\ref{MatrixEntries}), we are summing over all of $\Lambda$,
	while only the terms corresponding to $\Lambda_3$ are being
	considered there. In order to fix it,
	notice that $x$ does not lie in the finite set
	$\Set{\theta_{s^{\ast}}(y)}{s \in \Lambda_2}$, so we may
	choose $\varphi \in \Lc{X}$
	such that $\varphi(x)=1$ and
	$\varphi \big( \theta_{s^{\ast}}(y) \big)=0$
	for all $s \in \Lambda_2$.
	
	Let $\xi':=\pi(\varphi)\xi$ and notice that
	$$
	q_y \Big( \pi(f_s) \sigma_s \xi' \Big)
	= q_y \Big( \pi(f_s) \sigma_s \pi(\varphi)\xi \Big)
	= q_y \Big( \pi(f_s \balpha_s(\varphi)) \sigma_s \xi \Big)
	\overset{(\ref{joineigenvalue})}{=} f_s(y) \balpha_s(\varphi)|_y \, q_y(\sigma_s \xi).
	$$
	If $s \in \Lambda_1$, then the fact that $f_s$ is supported on
	$\ran{s}$ implies
	that $f_s(y)=0$, so the above expression vanishes. Moreover,
	if $s \in \Lambda_2$, then
	$$
	\balpha_s(\varphi)|_y = \varphi \big( \theta_{s^{\ast}}(y) \big) = 0,
	$$
	and the above expression vanishes again.
	From this we conclude that, for all $s \in \Lambda_1 \cup \Lambda_2$, we have
	\begin{equation}\label{MeanwhileClaim}
	q_y \Big(\pi(f_s)\sigma_s\xi' \Big) = 0.
	\end{equation}
	
	By noticing that
	$$
	q_x \Big( \pi(\varphi)\xi \Big)
	\overset{(\ref{joineigenvalue})}{=} \varphi(x) q_x(\xi) = q_x(\xi),
	$$
	and combining (\ref{MeanwhileClaim}) with (\ref{HypothesisPiVanishes}), we then have
	\begin{align*}
	0
	& = q_y \bigg( \sum_{s \in \Lambda} \pi(f_s) \sigma_s \xi' \bigg)
	= q_y \bigg( \sum_{s \in \Lambda_1} \pi(f_s)\sigma_s \xi' \bigg)
	+ q_y \bigg( \sum_{s \in \Lambda_2} \pi(f_s) \sigma_s\xi' \bigg)
	+ q_y \bigg( \sum_{s \in \Lambda_3} \pi(f_s) \sigma_s\xi' \bigg) \\
	& = q_y \bigg( \sum_{s \in \Lambda_3} \pi(f_s) \sigma_s\xi' \bigg)
	\overset{(\ref{WellDef})}{=}
			q_y \bigg( \sum_{s \in \Lambda_3} \pi(f_s) \sigma_s\xi \bigg)
	\overset{(\ref{MatrixEntries})}{=} \Big( (\Pi \times U)(b) q_x(\xi) \Big)_y.
	\end{align*}
	This shows that $(\Pi \times U)(b)$ vanishes on $\bigoplus_{x \in X} V_x$,
	and hence
	the proof is concluded.
\end{proof}

This result is fundamental for our study of ideals
in $\Lcs$. The method we shall adopt will be to start with any ideal
$J \trianglelefteq \Lcs$, and then use (\ref{nondegrep}) and
(\ref{localunits}) to find a representation $\pi$, as above, such
that $\ker(\pi)=J$. By (\ref{ThirdNullSpace}) we may replace $\pi$ by
$\discr$ acting on $\bigoplus_{x \in X} V_x$, without affecting null
spaces, and it will turn out that the latter decomposes as a direct
sum of very straightforward
sub-representations, which we will now describe.

\begin{proposition}\label{OrbitInvariant}
	Given any $x$ in $X$, we have that
	$$
	\bigoplus_{y \in \orb(x)} V_y
	$$
	is invariant under $\discr$.
\end{proposition}

\begin{proof}
	By (\ref{Ufunctoriality}.ii), for every $s \in S$, this
	space is invariant under $U_s$.
	It is also invariant under $\Pi(f)$, for every $f \in \Lc{X}$,
	since in fact each $V_y$ has this property.
	The invariance under $\discr$ then follows.
\end{proof}

We shall now study the representation obtained by restricting
$\discr$ to the invariant space mentioned above.

\begin{definition} \label{DefineRepOrbit}
	Given $x$ in $X$, we shall denote the invariant subspace referred to in
	(\ref{OrbitInvariant}) by $W_x$, while the representation of $\Lcs$
	obtained by restricting $\discr$ to $W_x$ will be denoted by $\rho_x$.
\end{definition}

If $R \subseteq X$ is a system of representatives for the orbit
relation in $X$, namely, if $R$ contains exactly one point of
each orbit relative to the action of $S$ on $X$, notice that
$$
\bigoplus_{y \in X} V_y = \bigoplus_{x \in R} W_x,
$$
while the restriction of $\discr$ to $\bigoplus_{y \in X} V_y$ is equivalent to
$\bigoplus_{x \in R} \rho_x$.

Before we state the main result of this section we should recall that right
after the proof of (\ref{localunits}) we fixed an arbitrary ideal
$J \trianglelefteq \Lcs$, which incidentally has been forgotten ever since.

\begin{theorem}\label{JIsIntersRho}
	Let $J$ be an arbitrary ideal of $\Lcs$, and let $\pi$ be a non-degenerate
	representation of $\Lcs$, such that $J=\ker(\pi)$. Considering the
	representations $\rho_x$ constructed above, we have
	$$
	J = \bigcap_{x \in R} \ker(\rho_x),
	$$
	where $R \subseteq X$ is any system of representatives for the orbit relation in $X$.
\end{theorem}

\begin{proof}
	The null space of $\pi$ coincides with the null space of the restriction of $\discr$
	to $\bigoplus_{x \in X} V_x$ by (\ref{ThirdNullSpace}). Since the latter
	representation is equivalent to the direct sum of the $\rho_x$, as seen above,
	the conclusion is evident.
\end{proof}

\subsection{The representations $\rho_x$}

In this section we are going to maintain all standing hypothesis of the previous section,
such as the ideal $J \trianglelefteq \Lcs$ and the representation
$\pi: \Lcs \to L(V)$ fixed there.

The usefulness of Theorem (\ref{JIsIntersRho}) in describing $J$
relies in our ability to describe the ideals $\ker(\rho_x)$
mentioned there. The good news is that the representations $\rho_x$ are
induced from representations of isotropy group algebras.
The main goal of this section is to prove that this is indeed the case.

Initially, notice that, if $x \in X$ and $s,t \in \till_x$ are
such that
$\germ{s}{x} = \germ{t}{x}$, then there exists $e \in E(S)$ such
that $x \in X_e$ and $se=te$. Hence, for $\eta \in V_x$,
we have
\begin{align}\label{welldefined}
	U_s(\eta)
	& = \mu_s^x(\eta)
	\overset{(\ref{mueident})}{=} \mu_s^x(\mu_e^x(\eta))
	\overset{(\ref{mufunctoriality})}{=} \mu_{se}^x (\eta) \nonumber \\
	& = \mu_{te}^x(\eta) 
	\overset{(\ref{mufunctoriality})}{=} \mu_t^x(\mu_e^x(\eta))
	\overset{(\ref{mueident})}{=} \mu_t^x(\eta)
	= U_t(\eta).
\end{align}

Our next result refers to the behavior of the operators $U_s$
when $\germo{s}$ lies in $G_x$.

\begin{proposition} \label{EnterKhModule}
	Fixing $x$ in $X$, let $G_x$ be the isotropy group of $x$. Then, for each
	$\germo{s}$ in $G_x$, we have that $V_x$ is invariant under $U_s$. 
	Moreover, the restriction of $U_s$ to $V_x$ is an invertible operator
	and the correspondence
	$$
	\germo{s} \in G_x \mapsto \restr{U_s}{V_x} \in GL(V_x)
	$$
	is a group representation.
	\end{proposition}

\begin{proof}
	Its well definiteness follows from (\ref{welldefined}). The
	remaining statements are
	immediate consequence of (\ref{Ufunctoriality}).
\end{proof}

The representation of $G_x$ on $V_x$ referred to in the above
Proposition may be integrated to a representation of $KG_x$, which
in turn makes $V_x$ into a left $KG_x$-module. Applying the
machinery of Section \ref{induction}, we may then form the induced
module $M_x \otimes V_x$, as in (\ref{InducedModule}),
which we may also view as a representation of $\Lcs$ on $M_x \otimes V_x$.

\begin{theorem}\label{RhoIsInduced}
For each $x$ in $X$, we have that $\rho_x$ is equivalent to the
representation induced from the left $KG_x$-module $V_x$, as
described above.
\end{theorem}

\begin{proof}
	Recalling from (\ref{DefineRepOrbit}) that $\rho_x$ acts on
	$$
	W_x = \bigoplus_{y \in \orb (x)} V_y,
	$$
	and that $M_x$ is a right $KG_x$-module, and viewing $V_x$ as a left
	$KG_x$-module via the representation mentioned in (\ref{EnterKhModule}),
	we claim that $T: M_x \times V_x \to W_x$ given by
	$$
	T \bigg( \sum_{ \germo{s} \in L_x} c_{\germo{s}} \delta_{\germo{s}},\xi \bigg)
	= \sum_{\germo{s} \in L_x} c_{\germo{s}} U_s(\xi)
	$$
	is a well-defined, balanced, bilinear map.
	
	Indeed, it is well defined by (\ref{welldefined}) and clearly bilinear.
	Moreover, for every $\germo{s} \in L_x$,
	$\germo{t} \in G_x$ and $\xi$ in $V_x$, we have
	$$
	T(\delta_{\germo{s}}\delta_{\germo{t}},\xi )
	= T(\delta_{\germo{st}}, \xi)
	= U_{st}(\xi)
	\overset{(\ref{Ufunctoriality})}{=} U_s \big( U_t(\xi) \big)
	= T \big(\delta_{\germo{s}}, U_t(\xi) \big)
	= T(\delta_{\germo{s}}, \delta_{\germo{t}} \cdot \xi).
	$$
	Therefore, there exists a unique linear map $\tau :M_x \otimes V_x \to W_x$,
	such that $\tau(\delta_{\germo{s}} \otimes \xi) = U_s(\xi)$. We shall
	next prove that $\tau$
	is an isomorphism by exhibiting an inverse for it.
	
	With this in mind, let $R_x \subseteq L_x$ be a total system of
	representatives for left $G_x$-classes. Thus, if $y$ is in the orbit of
	$x$, there exists a
	unique $\germo{r} \in R_x$ such that $\theta_r(x)=y$, so that
	$U_{r^{\ast}}$ maps $V_y$
	onto $V_x$, by (\ref{Ufunctoriality}). We therefore let
	$$
	\upsilon_y: V_y \to M_x \otimes V_x
	$$
	be given by $\upsilon_y(\xi) = \delta_{\germo{r}} \otimes U_{r^{\ast}}(\xi)$,
	for every
	$\xi $ in $V_y$. Putting all of the $\upsilon_y$ together, let
	$$
	\upsilon: W_x = \bigoplus_{y \in \orb(x)} V_y \longrightarrow M_x \otimes V_x
	$$
	be the unique linear map coinciding with $\upsilon_y$ on $V_y$,
	for every $y$ in $\orb(x)$.

	We claim that $\upsilon$ is the inverse of $\tau$. To see this,
	let $\germo{s}$
	be any element in
	$L_x$, and let $\xi$ be picked in $V_x$ arbitrarily. Let $\germo{r} \in R_x$
	be such that $\theta_s(x) = \theta_r(x)$. Then, $r^{\ast}s$ lies in
	$\tilh_x$ and $U_s(\xi) \in V_y$, where $y:= \theta_s(x) = \theta_r(x)$.
	We then have
	\begin{align*}
		\upsilon \bigg( \tau(\delta_{\germo{s}} \otimes \xi) \bigg)
		& = \upsilon \big( U_s(\xi) \big)
		= \delta_{\germo{r}} \otimes U_{r^{\ast}} \big( U_s(\xi) \big)
		= \delta_{\germo{r}} \otimes U_{r^{\ast}s}(\xi)
		= \delta_{\germo{r}} \otimes \delta_{\germo{r^{\ast}s}} \cdot \xi \\
		& = \delta_{\germo{r}} \delta_{\germo{r^{\ast}s}} \otimes \xi
		= \delta_{\germo{rr^{\ast}s}} \otimes \xi
		= \delta_{\germo{s}} \otimes \xi.
	\end{align*}
	On the other hand,
	given any $y$ in $\orb(x)$ and $\xi \in V_y$, write $y=\theta_r(x)$,
	for $\germo{r} \in R_x$, and notice that
	$$\tau \big(\upsilon (\xi) \big)
	=\tau \bigg(\delta_{\germo{r}} \otimes U_{r^{\ast}}(\xi)\bigg)
	= U_r \big(U_{r^{\ast}} (\xi) \big)
	= U_{rr^{\ast}}(\xi)
	= \xi.
	$$
	
	Therefore $\tau$ is indeed an isomorphism between the $K$-vector spaces
	$M_x \otimes V_x$ and
	$W_x$. We will next prove that $\tau$ is equivariant for the respective
	actions of
	$\Lcs$, which amounts to say that it is linear as a map between left
	$\Lcs$-modules.
	For this, given $t \in S$, and $f \in \Lc{\ran{t}}$, we must prove that
	\begin{equation}\label{ThisIsCovariance}
		\tau \bigg( (f \Delta_t) \delta_{\germo{s}} \otimes \xi \bigg)
		= \rho_x(f \Delta_t) \bigg( \tau(\delta_{\germo{s}} \otimes \xi ) \bigg),
	\end{equation}
	for all $\germo{s} \in L_x$ and all $\xi \in V_x$.
	
	Notice that, if $\germo{s} \in L_x$ and $\xi \in V_x$, then
	the left-hand side of (\ref{ThisIsCovariance}) equals
	$$
	\tau \bigg( (f \Delta_t ) \delta_{\germo{s}} \otimes \xi \bigg)
	= \bool{ts \in \till_x} f \big(\theta_{ts}(x) \big)
			\tau \big( \delta_{\germo{ts}} \otimes \xi \big)
	= \bool{ts \in \till_x} f \big(\theta_{ts}(x) \big) U_{ts}(\xi)
	$$
	while the right-hand side becomes
	\begin{equation}\label{RhofDgOnTau}
		\rho_x(f \Delta_t) \bigg( \tau( \delta_{\germo{s}} \otimes \xi) \bigg)
		= \Pi(f) U_t U_s(\xi)
		= \Pi(f) U_{ts}(\xi).
	\end{equation}
	Since $\xi$ lies in $V_x$, recall from (\ref{Ufunctoriality})
	that $U_{ts}$ vanishes on $V_x$, unless $ts$
	lies $\till_x$,
	in which case $U_{ts}$ maps $V_x$ bijectively onto
	$V_{\theta_{ts}(x)}$. Hence, (\ref{RhofDgOnTau}) becomes 
	$$
	\rho_x(f \Delta_t) \bigg( \tau( \delta_{\germo{s}} \otimes \xi) \bigg)
	= \Pi(f) U_{ts}(\xi)
	= \bool{ts \in \till_x} f \big( \theta_{ts}(x) \big) U_{ts}(\xi)
	$$
	because, $\Pi(f)$ acts on $V_{\theta_{ts}(x)}$ by scalar
	multiplication by
	$f \big( \theta_{ts}(x) \big)$,
	according to (\ref{eigenvalue}).
	
	This proves (\ref{ThisIsCovariance}), so $\tau$ is indeed covariant.
\end{proof}

Summarizing what we have done so far, the following is the main
result of this work.

\begin{theorem}\label{MainResult}
	Let $(\theta, S, X)$ be an ample system and $\Lcs$ be the corresponding
	crossed product algebra over a field $K$.
	Then, every ideal $J \trianglelefteq \Lcs$ is the intersection of
	ideals induced from isotropy groups.
\end{theorem}

\begin{proof}
	Let $R \subseteq X$ be a system of representatives for the orbit relation on $X$.
	Using (\ref{JIsIntersRho}) we may write $J$ as the intersection of the null
	spaces of the $\rho_x$, for $x$ in $R$, while (\ref{RhoIsInduced})
	tells us that $\rho_x$ is equivalent to the representation induced from a representation
	of the isotropy group at $x$. The null space of $\rho_x$ is therefore induced from an ideal
	in the group algebra of the said isotropy group by (\ref{AnnInducedVsIndIdeal}),
	whence the result.
\end{proof}

Next proposition goes in the way of describing explicitly a
given ideal $J \trianglelefteq \Lcs$ as the intersection of induced
ideals.

\begin{proposition}\label{IntersectionDescription}
	Under the assumptions of (\ref{MainResult}), choose a system $R$
	of representatives for the orbit relation on $X$. For each $x$ in
	$R$, let $G_x$ be the isotropy group at $x$, and let
	$$
	\Gamma_x:\Lcs \to KG_x
	$$
	be as in (\ref{IntroduceGamma}). Then, given any ideal
	$J \trianglelefteq \Lcs$ we have that $\Gamma_x(J)$ is an
	admissible ideal of $KG_x$, and
	$$
	J = \bigcap_{x \in R} \ind_x(\Gamma_x(J)).
	$$	
\end{proposition}

\begin{proof}
	Let ${I'}_x:=\Gamma_x(J)$. That each ${I'}_x$ is an
	admissible ideal follows at once from (\ref{GammaAdmissible}).	
	For each $x$ in $R$, let $I_x$ be the null space of the
	representation $\rho_x$
	referred to in the proof of (\ref{MainResult}), so that
	$$
	J = \bigcap_{x \in R} \ind_x(I_x).
	$$
	Observe that for each $x \in R$, we have
	$$
	{I'}_x
	= \Gamma_x(J)
	= \Gamma_x \bigg( \bigcap_{y \in R} \ind_y({{I}_y}) \bigg)
	\subseteq \Gamma_x \Big( \ind_x(I_x) \Big)
	\overset{(\ref{youaretheone})}{\subseteq} I_x.
	$$
	Consequently
	$\ind_x({I'}_x) \subseteq \ind_x({I}_x)$, whence
	$$
	\bigcap_{x \in R} \ind_x({I'}_x)
	\subseteq \bigcap_{x \in R} \ind_x({I}_x)
	= J.
	$$

	On the other hand, we have by (\ref{largestideal}) that $\ind_x({I'}_x)$
	is the largest among the ideals of $\Lcs$ mapping into ${I'}_x$
	under $\Gamma_x$. Since
	$\Gamma_x(J)={I'}_x$, by definition, we have that $J$ is among
	such ideals, so
	$J \subseteq \ind_x({I'}_x)$, and then
	$$
	J \subseteq \bigcap_{x \in R} \ind_x({I'}_x),
	$$
	concluding the proof.
\end{proof}

%% file: Steinberg.tex
\section{Application for Steinberg algebras}

In this section, we prove that every Steinberg algebra associated with an
ample groupoid can be realized as an inverse semigroup crossed product algebra
of the form $\Lcs$.
For the task, we first show that the Steinberg algebra associated with the
groupoid of germs of an ample dynamical
system is isomorphic to the crossed product
algebra as a consequence of the theory we have developed so far.
Then, combining this with an Exel's result in \cite{ExelCombinatorial},
we get the promised result.

We assume the reader is familiar with the notion of topological groupoids
and in particular with its basic notations: a groupoid is usually denoted
by $\G$, its unit space by $\G^{(0)}$, and the set of composable pairs by
$\G^{(2)}$. The source and range maps are denoted by $d$ and $r$,
respectively.

An \emph{\'etale} groupoid is a topological groupoid $\G$, whose unit
space $\G^{(0)}$ is locally compact and Hausdorff in the relative topology,
and such that the range map $r: \G \to \G^{(0)}$ is a local homeomorphism
\cite{ExelCombinatorial}.

A very important class of \'etale groupoids is that of \emph{ample}
groupoids \cite{patersonbook}. An \'etale groupoid is called ample if the
compact bisections form a basis for its topology, where a bisection is an
open subset $U \subseteq \G$ such that the restrictions of $d$ and $r$ to
$U$ are injective.

If $\G$ is an ample groupoid, then the Steinberg algebra $\A{\G}$ is defined
as the space of all $K$-valued functions on $\G$ spanned by functions
$f:\G \to K$ such that:
\begin{itemize}
	\item There is an open Hausdorff subspace $V$ in $\G$ so that $f$ vanishes
	outside $V$; and
	\item $\restr{f}{V}$ is locally constant with compact support;
\end{itemize}
with pointwise sum and convolution product.

Note that if $\G$ is not Hausdorff, then $\A{\G}$ will contain discontinuous
functions. The reader is referred to \cite{Lisa} and \cite{Steinberg} for
detailed treatment in the subject.

\subsection{Induction process for Steinberg algebras}

\textbf{From now on, we fix an ample groupoid $\G$ and its
associated Steinberg algebra $\A{\G}$.}

In \cite{Steinberg}, Steinberg also develops a theory of induction of modules
from isotropy groups.

For a point $x \in \G^{(0)}$ he considers:

\begin{equation}\label{LHOrb}
\begin{array}{rcl}
\mL_x & := & \Set{\gamma \in \G}{d(\gamma)=x}, \\
\mH_x & := & \Set{\gamma \in \G}{d(\gamma)=r(\gamma)=x}, \\
\orbG(x) & := & \Set{r(\gamma)}{\gamma \in \mL_x}.
\end{array}
\end{equation}

Moreover, he considers $\mM_x$ as the free $K$-module with basis $\mL_x$. Since
$\mL_x \mH_x \subseteq \mL_x$, there is a natural right $K\mH_x$-module structure
on $\mM_x$. Moreover, $\mM_x$ is $\A{\G}$-$K\mH_x$-bimodule, where the left
structure is such that
\begin{equation}\label{AkAction}
	f \cdot \delta_{\nu}
	= \sum_{\gamma \in L} f(\gamma \nu^{-1}) \delta_{\gamma}.
\end{equation}

In this fashion, if $x \in \G^{(0)}$ and $V$ is a left $K\mH_x$-module, then the
left $\A{\G}$-module induced by $V$ is defined by
$$\ind_{x}(V) := \mM_x \otimes_{K\mH_x} V.$$

We strongly encourage \cite{Steinberg} for more details in the subject.

To introduce the notion of an induced ideal, we first talk about a map that
will play a crucial role in the road to our ambitions. This is a version
of $\Gamma_x$ to the actual context. For each $x \in \G^{(0)}$, consider
the map $\mGamma_x: \A{\G} \to K\mH_x$ given by
\begin{equation}\label{IntroduceGammaGG}
\mGamma_x \left( f \right) = \sum_{\gamma \in \mH_x} f(\gamma) \delta_{\gamma}.
\end{equation}

We then have the following proposition.

\begin{proposition}\label{quasihomoG}
	Let $x \in \G^{(0)}$ and let $U$ be a compact open bisection such that
	$\mH_x \cap U \neq \emptyset$. Then, for every $f \in \A{\G}$, we have
	$$
	\mGamma_x(uf)=\mGamma_x(u)\mGamma_x(f)
	\quad \text{ and } \quad
	\mGamma_x(fu)=\mGamma_x(f)\mGamma_x(u),
	$$
	where $u$ stands for the characteristic function of $U$.
\end{proposition}

\begin{proof}
	Notice that, since $U$ is a bisection such that
	$\mH_x \cap U \neq \emptyset$, there exists an unique element
	$\nu \in \mL_x \cap U$. Hence, if $\gamma \in \mH_x$,
	we have
	$$
	fu(\gamma)
	= \sum_{\mu \in \mL_x} f(\gamma \mu^{-1}) u(\mu)
	= f(\gamma \nu^{-1})
	$$
	and, therefore,
	$$
	\mGamma_x(fu)
	= \sum_{\gamma \in \mH_x} (fu)(\gamma) \delta_{\gamma}
	= \sum_{\gamma \in \mH_x} f(\gamma \nu^{-1}) \delta_{\gamma}.
	$$
	On the other hand,
	$$
	\mGamma_x(f)\mGamma_x(u)
	= \bigg( \sum_{ \gamma \in \mH_x} f(\gamma) \delta_{\gamma} \bigg) \delta_{\nu}
	= \sum_{ \gamma \in \mH_x} f(\gamma) \delta_{\gamma \nu}
	= \sum_{\gamma \in \mH_x} f(\gamma \nu^{-1}) \delta_{\gamma}.
	$$
	Similarly, we can show $\mGamma_x(uf)=\mGamma_x(u)\mGamma_x(f)$, concluding
	the proof.
\end{proof}

\begin{proposition}\label{GammaIdealG}
	Let $J \trianglelefteq \A{\G}$ be an ideal
	and $x \in \G^{(0)}$. Then, $\mGamma_x(J)$ is an ideal
	in $K\mH_x$. 
\end{proposition}

\begin{proof}
	Let $a \in \mGamma_x(J)$ and $b=\delta_{\gamma} \in K\mH_x$ for some
	$\gamma \in \mH_x$. Then, there exists $f \in J$ such that $\mGamma_x(f)=a$.
	Notice that, by choosing a compact open bisection $U$ containing $\gamma$,
	we have
	$$
	ab
	= a\delta_{\gamma}
	= \mGamma_x(f)\mGamma_x(1_U)
	\overset{(\ref{quasihomoG})}{=} \mGamma_x(f 1_U) \in \mGamma_x(J).
	$$
	By linearity, we deduce that $ab \in \mGamma_x(J)$ for arbitrary
	$b \in K\mH_x$ and similarly, we can show that $ba \in \mGamma_x(J)$.
\end{proof}

The next definition should not be strange to the reader at this point.
Indeed, we shall see that, if $I$ is the annihilator of $V$ in
$K\mH_x$, then $\Ind{x}{I}$, as defined above, is the annihilator of
$\mM_x  \otimes V$ in $\A{\G}$. Actually, we could verify it right
now. However, since it shall become clear soon, we choose to spare the
work.

\begin{definition}\label{InducedIdealG}
	Let $x \in \G^{(0)}$. Given any ideal $I \trianglelefteq K\mH_x$,
	we define
	$$
	\Ind{x}{I}
	:= \Set{f \in \A{\G}}{\mGamma_x(ufv) \in I, \forall u,v \in \A{\G}},$$
	and call it the \emph{ideal induced by} $I$.
\end{definition}


\subsection{Universal property for the Steinberg algebra associated with a
	groupoid of germs}

In section 4 of \cite{ExelCombinatorial}, Exel introduced the groupoid of
germs associated with an action of an inverse semigroup on a locally compact
Hausdorff topological space. For the convenience of the reader and to
introduce some notations, we review briefly this theory. However, the
interested reader is strongly encouraged to read \cite{ExelCombinatorial}
for more details in the subject.

For the moment, let $(\theta,S,X)$ be a topological dynamical system. The
groupoid of germs, which we denote $S \ltimes_{\theta} X$ (or simply
$S \ltimes X$ when the action is implicit in the context), as a set, is the quotient
of the set
$$\left\{ (s,x) \in S \times X : x \in \dom{s} \right\}$$
by the equivalence relation that identifies two pairs $(s,x)$ and $(t,y)$
if and only if $x=y$ and there exists an idempotent $e \in E(S)$ such that
$x \in X_e$ and $se=te$. We denote by $\left[ s,x \right]$ the equivalence
class of $(s,x)$ and call it the germ of $s$ at $x$. The inversion is given
by $\left[ s,x \right]^{-1}=\left[s^{\ast}, \theta_s(x) \right]$ and the
multiplication is given by defining $\left[ s,x \right] . \left[ t,y \right]$
if and only if $x=\theta_t(y)$, in which case the product is
$\left[ st,y \right]$.

A basis for the topology of $S \ltimes X$ is given by
$\Theta(s,U) = \left\{ \germ{s}{x} \in S \ltimes X : x \in U \right\}$
where $s \in S$ and $U \subset \dom{s}$ is an open set. Furthermore, the
map $x \in U \mapsto \left[s,x \right] \in \Theta(s,U)$ is a homeomorphism,
where $\Theta(s,U)$ carries the topology induced from $S \ltimes X$.

The unit space is formed by elements $\left[ e,x \right]$ with $e \in E(S)$
and $x \in X_e$, and the map
\begin{equation}\label{Xunitspace}
	\germ{e}{x} \mapsto x
\end{equation}
gives a homeomorphism
between the unit space of $S \ltimes X$ and $X$. So, from
now on, we identify the unit space with $X$ and, with such an
identification, we have that $d(\left[ s,x \right])=x$ and
$r(\left[ s,x \right])=\theta_s(x)$ are the domain and range maps,
respectively.

In this setting, $S \ltimes X$ is an \'etale groupoid and, for
each $s \in S$ and each open subset $U$ of $\dom{s}$, $\Theta(s,U)$ is a
bisection. We will use the shorthand notation $\Theta_s$ for the bisection
$\Theta(s,\dom{s})$.

Furthermore, if $(\theta,S,X)$ is an ample dynamical system, the
unit space of $S \ltimes X$ is totally disconnected and, hence,
the collection of all
compact bisections forms a basis for the topology of $S \ltimes X$,
according to
Proposition (4.1) of \cite{ExelAmple}. This amounts to say that
$S \ltimes X$ is an ample groupoid.
Therefore, we can build the Steinberg algebra $\A{S \ltimes X}$
associated with $S \ltimes X$.

It worths to mention that the groupoid of germs does not need be Hausdorff.
The interested reader is referred to \cite{ExelPardo} for a
characterization of Hausdorffness for the groupoid of germs.

\textbf{From now on, we fix an ample dynamical system $(\theta,S,X)$
and the Steinberg algebra $\A{S \ltimes X}$ associated with the
groupoid of germs $S \ltimes X$.}

We will denote by $d_s$ and $r_s$ the restrictions of the source and range
maps to $\Theta_s$, respectively. The maps $d_s$ and $r_s$ are homeomorphisms
onto their images $\dom{s}$ and $\ran{s}$, respectively. Notice that, if
$\varphi \in \Lc{\ran{s}}$, then the composition $\varphi \circ r_s$ is a
compactly supported locally constant function on $\Theta_s$. So, we shall
also see $\varphi \circ r_s$ as a function in $\A{S \ltimes X}$ by extending
them to
be zero outside $\Theta_s$, which we shall denote by $\varphi \varDelta_s$.

For each $s \in S$ and each $f \in \Lc{\dom{s}}$ we will denote by $\alpha_s(f)$
the element of $\Lc{\ran{s}}$ given by
$$\restr{\alpha_s(f)}{x} = f( \theta_{s^{\ast}}(x) ) \text{ for all } x \in X.$$

In this context, we have the result.

\begin{proposition}\label{CompatibleProduct}
	Given $f \in \Lc{\ran{s}}$ and $g \in \Lc{\ran{t}}$, we have
	$$
	f \varDelta_s \ast g \varDelta_t
	= \alpha_s(\alpha_{s^\ast}(f)g) \varDelta_{st}.
	$$
\end{proposition}

\begin{proof}
	Let's prove initially that $\Theta_s \Theta_t = \Theta_{st}$. Indeed, it
	is clear that $\Theta_s \Theta_t \subseteq \Theta_{st}$. Conversely, let
	$\germ{st}{y} \in \Theta_{st}$ and notice that we must have
	$$
	y \in X_{(st)^{\ast}st}
	= X_{t^{\ast}s^{\ast}st}
	= \theta_{t^{\ast}}(\dom{s} \cap \ran{t}).
	$$
	This means that there exists $x \in \dom{s} \cap \ran{t}$ such that
	$y = \theta_{t^{\ast}}(x)$. In particular, $x \in \dom{s}$ and
	$y \in \ran{t}$, which implies that $\germ{s}{x} \in \Theta_s$ and
	$\germ{t}{y} \in \Theta_t$. Then, since $x = \theta_t(y)$, we have
	$$
	\germ{st}{y}
	= \germ{s}{x} \germ{t}{y} \in \Theta_s \Theta_t,
	$$
	concluding the initial assumption.
	
	Since $f \in \Lc{\ran{s}}$ and $g \in \Lc{\ran{t}}$, we have
	$$
	\supp (f \varDelta_s \ast g \varDelta_t)
	\subseteq \Theta_s \Theta_t
	= \Theta_{st}.
	$$
	
	Notice now that
	\begin{align*}
		\bigg( f \varDelta_s \ast g \varDelta_t \bigg) (\germ{st}{y})
		& = (f \varDelta_s)(\germ{s}{x}) (g \varDelta_t)(\germ{t}{y})
		= f \big( \theta_s(x) \big) g \big( \theta_t(y) \big) \\
		& = f \big( \theta_{st}(y) \big) g \big( \theta_t(y) \big)
		= \bigg( \alpha_s \big( \alpha_{s^{\ast}}(f)g \big)
							\varDelta_{st} \bigg)(\germ{st}{y}).
	\end{align*}
\end{proof}

This proposition is important to establish the isomorphism between
the crossed product algebra and the Steinberg algebra associated with
the groupoid of germs, which is our aim now.

Before we proceed, for the sake of understanding, let's take a pause
to interpret the objects of last subsection in the context of the
groupoid of germs $S \ltimes X$.

Laying the groundwork, from the point of view of a groupoid of germs,
having in mind the identification done in (\ref{Xunitspace}), notice
that the sets defined in (\ref{LHOrb}) can be interpreted as
\begin{equation}\label{LHOrb2}
\begin{array}{rcl}
\mL_x & := & \Set{\germ{s}{x}}{x \in \dom{s}} \\
\mH_x & := & \Set{\germ{s}{x}}{x \in \dom{s} \text{ and } \theta_s(x)=x} \\
\orbG(x) & := & \Set{\theta_s(x)}{x \in \dom{s}}
\end{array}
\end{equation}

In relation to the definitions in (\ref{LHOrb0}), notice that the map
$$\germ{s}{x} \in \mL_x \mapsto \germ{s}{x} \in L_x$$
is a bijection which restricts to an isomorphism between $\mH_x$ and $G_x$.
Hence, the isotropy in each sense coincide.

From now on, we identify these objects as well as the isotropy group algebra
and, hence, we abolish the bold notation in (\ref{LHOrb2}).

We now have the tools to the establish an isomorphism between the inverse
semigroup crossed product algebra associated with an ample system
$(\theta,S,X)$
and the Steinberg algebra associated with the respective groupoid of germs
$S \ltimes X$. The reader may compare with Theorem 5.4 of
\cite{CordeiroBeuter}.

\begin{theorem}\label{isomorphism1}
	Let $(\theta,S,X)$ be an ample dynamical system, $\alpha$ the action
	of $S$ on $\Lc{X}$ given by (\ref{inducedaction}) and $S \ltimes X$ the
	associated groupoid of germs. Then 	$\Lc{X} \rtimes_{\alpha} S$ is
	isomorphic to  $\A{S \ltimes X}$.
\end{theorem}

\begin{proof}
	Let $\mathcal{B}^{\theta}$ be the semi-direct product
	bundle associated with $(\theta,S,X)$, as in (\ref{Triple}). Consider,
	for each $s \in S$, the map
	$$
	\pi_s:
	f \bmdelta_s \in B_s \mapsto f \varDelta_s \in \A{S \ltimes X}.
	$$
	Then, by (\ref{CompatibleProduct}), $\left\{ \pi_s \right\}_{s \in S}$ is
	a pre-representation of the semi-direct product bundle
	$\mathcal{B}^{\theta}$ in
	$\A{S \ltimes X}$. Furthermore, if $s \leq t$ in $S$ and $f$ lies in
	$\Lc{\ran{s}}$,
	then it is easy to see that $f \varDelta_t$ vanishes outside $\Theta_s$
	and, then, coincides with $f \varDelta_s$. This amounts to say that
	$\left\{ \pi_s \right\}_{s \in S}$ is a representation of
	$\mathcal{B}^{\theta}$ in $\A{S \ltimes X}$.
	
	By Proposition (\ref{UniversalProp}), there exists an epimorphism
	$\Phi: \Lc{X} \rtimes_{\alpha} S \to \A{S \ltimes X}$ such that
	$\Phi(f \Delta_s) = f \varDelta_s$.
	
	To show that $\Phi$ is injective, we claim first that the diagram
	\begin{equation}\label{diagram}
	\begin{tikzcd}
	\Lcs \arrow{r}{\Phi} \arrow[swap]{dr}{\Gamma_x} & \A{S \ltimes X}
	\arrow{d}{\mGamma_x} \\
	& KG_x
	\end{tikzcd}
	\end{equation}
	is commutative for every $x \in X$. Indeed, let $f \in \Lc{\ran{s}}$
	and notice that, since $f \varDelta_s$ is a function supported on
	$\Theta_s$ and the latter is a bisection, there exists at most one
	element in $\Theta_s \cap G_x$.
	Actually, if $x \in \dom{s}$ and $\theta_s(x)=x$ we must have
	$\Theta_s \cap G_x = \left\{ \germ{s}{x} \right\}$ and, otherwise, we
	must have $\Theta_s \cap G_x = \emptyset$. Therefore,
	$$
	(\mGamma_x \circ \Phi)(f \Delta_s)
	= \mGamma_x(f \varDelta_s)
	= [x \in \dom{s}, \theta_s(x)=x] f(x) \delta_{\germ{s}{x}}
	= [s \in \tilh_x] f(x) \delta_{\germ{s}{x}}
	= \Gamma_x(f \Delta_s).
	$$
	This concludes the diagram commutativity.
		
	Let $J \trianglelefteq \Lcs$ be the kernel of $\Phi$. Then, by the
	commutativity of the diagram we have
	$$\Gamma_x(J) = (\mGamma_x \circ \Phi) (J) = \mGamma_x(0) = 0,$$
	for every $x \in X$. Hence, by (\ref{IntersectionDescription})
	$$
	J
	= \bigcap_{x \in X} \ind_x \big( \Gamma_x(J) \big)
	= \bigcap_{x \in X} \ind_x \big( 0 \big)
	= \bigcap_{x \in X} \ind_x \big( \Gamma_x(0) \big)
	= 0,
	$$
	concluding the proof.
\end{proof}

There are some immediate consequences of this result.
First, note that $\A{S \ltimes X}$ inherits the universal property of $\Lcs$,
which we spell out. The reader is invited to compare with Theorem 4.27
of \cite{Steinberg}. On the one hand, Steinberg demands the groupoid of
germs
$S \ltimes X$ to be Hausdorff, on the other hand, the assumption
that $K$ is a field is relaxed by considering algebras over
a commutative ring with identity.

\begin{proposition}
	Let $(\theta, S , X)$ be an ample dynamical system and
	$S \ltimes X$ its associated groupoid of germs. Then,
	for any covariant representation $(\pi,\sigma)$ of
	$(\theta,S,X)$, there exists a non-degenerate representation
	$\pi \times \sigma$ of $\A{S \ltimes X}$ such that
	$$(\pi \times \sigma)(f \varDelta_s) = \pi(f) \sigma_s$$
	Furthermore, the mapping
	$$ (\pi,\sigma) \mapsto \pi \times \sigma $$
	gives a bijection between covariant representations
	of $(\theta, S, X)$ and non-degenerate representations
	of $\A{S \ltimes X}$.
\end{proposition}

\begin{proof}
	Just join (\ref{isomorphism1}), (\ref{disintegration0})
	and (\ref{integration2}).
\end{proof}

Furthermore, by the commutativity of diagram (\ref{diagram}), we see that
$\Phi$ maps $\Ind{x}{I}$ in the sense of Definition (\ref{InducedIdeal})
onto $\Ind{x}{I}$ in the sense of Definition (\ref{InducedIdealG}).

Moreover, $\Phi$ is compatible with the actions of $\Lcs$ and $\A{S \ltimes X}$
on $M_x$ in the sense that, the left $\A{\G}$-module
structure on $M_x$ induced
by $\Phi$ from (\ref{LcsAction}) is exactly the same structure as defined in \eqref{AkAction}. This immediately gives the following two
propositions.

\begin{proposition}
	Let $x \in \mathcal{G}^{(0)}$. If $I$ is the annihilator of
	$V$ in $KG_x$, then $\Ind{x}{I}$ is
	the annihilator of $M_x \otimes V$ in $\A{S \ltimes X}$.
\end{proposition}

This proposition makes sense to Definition (\ref{InducedIdealG}),
while the next one translates the main result for inverse
semigroup crossed product algebras to Steinberg algebras
(associated with groupoid of germs). But first, let us bring
the definition of admissible ideal from
(\ref{admissibleideal}) for the current context.

\begin{definition}\label{admissibleidealG}
	An ideal $I \trianglelefteq KG_x$ is said to be \emph{admissible} if $\Gamma_x(\ind_x(I))=I$.
\end{definition}

Then, we finally have the desired result.

\begin{proposition}
	
	Let $(\theta, S, X)$ be an ample system and $S \ltimes X$
	the associated groupoid of germs. 
	Choose a system $R$
	of representatives for the orbit relation on $X$. For each $x$ in
	$R$, let $G_x$ be the isotropy group at $x$, and let
	$$
	\Gamma_x:\A{S \ltimes X} \to KG_x
	$$
	be as in (\ref{IntroduceGammaGG}). Then, given any ideal
	$J \trianglelefteq \A{S \ltimes X}$ we have that
	$\Gamma_x(J)$ is an
	admissible ideal of $KG_x$, and
	$$
	J = \bigcap_{x \in R} \ind_x(\Gamma_x(J)).
	$$	
\end{proposition}
	

\subsection{Steinberg Algebras as crossed products}

In section 5 of \cite{ExelCombinatorial}, Exel presents an example of an inverse
semigroup action which is intrinsic to every \'etale groupoid.  We therefore fix
an \'etale groupoid $\G$ from now on and denote by $\SG$ the set of all bisections
in $\G$. It is well known that $\SG$ is an inverse semigroup under the operations
$$
UV = \left\{ uv \text{ : } u \in U, v \in V, (u,v) \in \G^{(2)} \right\}
\text{ and }
U^{\ast} = \left\{ u^{-1} : u \in U \right\},
$$
The idempotent semilattice of $\SG$ consist
precisely of the open subsets of $\G^{(0)}$.

Moreover, for any bisection $U$, its source $d(U)$ and range
$r(U)$ are
open subsets of
$\mathcal{G}^{(0)}$ and the maps
$d_U : U \to d(U)$ and $r_U : U \to d(U)$,
obtained by restricting $d$ and $r$, respectively, are homeomorphisms.
Hence, we can define a topological action
$\theta: \SG \to \mathcal{G}^{(0)}$ such that,
for each $U \in \SG$,
\begin{equation}\label{intrinsicaction}
\begin{array}{rccc}
\theta_U: & d(U) & \longrightarrow & r(U)\\
& x & \longmapsto & r_U \Big( d_U^{-1}(x) \Big).
\end{array}
\end{equation}
Notice that $\theta_U(x)=y$, if and only if there
exists some $u \in U$ such that $d(u)=x$ and $r(u)=y$.

Additionally, given any 
*-subsemigroup $S \subseteq \SG$, we may restrict $\theta$ to $S$, thus obtaining a
semigroup homomorphism
$$
\restr{\theta}{S}: S \to \mathcal{I}(X)
$$
which is an action of $S$ on $X$, provided (\ref{topaction}.\ref{nondegenerate})
can be verified.
The next result gives sufficient conditions for the groupoid of
germs for such an action to be equal to $\G$. 

\begin{proposition}\label{TwoGroupoids}
	Let $\G$ be an \'etale groupoid and let $S$ be a *-subsemigroup of
	$\SG$ such that
	\begin{enumerate}[(i)]
		\item $\G = \bigcup_{U \in S} U$, and
		\item for every $U,V \in S$, and every $u \in U \cap V$, there exists
		$W \in S$, such that $u \in W \subseteq U \cap V$.
	\end{enumerate}
	Then $\restr{\theta}{S}$ is an action of $S$ on $X=\G^{(0)}$,
	and the groupoid of germs for $\theta|_S$ is isomorphic to $\G$.
\end{proposition}

\begin{proof}
	Proposition (5.4) of \cite{ExelCombinatorial}.
\end{proof}

An interesting consequence of the above proposition is that, if $\G$ is an
ample groupoid, then the set of compact bisections of $\G$ are in the hypotheses
of (\ref{TwoGroupoids}). Hence, the groupoid of germs obtained by the restriction
of the action $(\theta,S,X)$ above to the *-subsemigroup of compact bisections is
isomorphic to the original groupoid $\G$. Moreover, this restriction forms an
ample action (with domains compact).

We could even add the unit space to the *-subsemigroup of compact bisections
and it would still satisfy the hypotheses of (\ref{TwoGroupoids}). In this
situation, the restriction of the action would still form an ample action and
the semigroup involved would have a unit. Indeed, this action can
be interpreted as the unitization of the former in the case that
$\G^{(0)}$ is not compact.

We shall denote by $S^a$ the *-subsemigroup of $\SG$ formed by
the compact bisections of $\G$.
The importance of the comments above arises
from the next proposition.

\begin{proposition}\label{SteinvsCrossedproduct}
	Let $\G$ be an ample groupoid and let $S$ be a *-subsemigroup of $\SG$
	satisfying 	the
	hypotheses of (\ref{TwoGroupoids}) and such that the restriction $\theta$
	to $S$ 	of the action of $\SG$ on $\G^{(0)}$ given by
	(\ref{intrinsicaction}) is ample. 	If $\alpha$ is the induced action of
	$S$ on $\Lc{\G^{(0)}}$, as in (\ref{inducedaction}), then
	$$\A{\G} \simeq \Lc{\G^{(0)}} \rtimes_{\alpha} S.$$
\end{proposition}

\begin{proof}
	Let $S \ltimes \G^{(0)}$ be the groupoid of germs for the given
	action of $S$ on $\G^{(0)}$. Applying (\ref{isomorphism1}), we conclude that
	$$\A{S \ltimes \G^{(0)}} \simeq \Lc{\G^{(0)}} \rtimes_{\alpha} S.$$
	By (\ref{TwoGroupoids}), $S \ltimes \G^{(0)} \simeq \G$. Hence,
	$$\A{\G} \simeq \Lc{\G^{(0)}} \rtimes_{\alpha} S,$$
	as desired.
\end{proof}

The reader is invited to compare this result with Corollary 5.6
of \cite{CordeiroBeuter} and Theorem 5.2 of \cite{Beuter}.
Summarizing, every Steinberg algebra associated to an ample groupoid can
be viewed as an inverse semigroup crossed product algebra. In
particular, by the comments immediately before Proposition
(\ref{SteinvsCrossedproduct}),
one may always choose $S$ to be $S^a$ or even $S^a$ added by
the unit space if it is desired to deal only with inverse semigroups
with a unit.

%% file: CENTRAL_FILE.bbl
\begin{thebibliography}{10}

\bibitem{effros1967}
EFFROS, E.~G.; HAHN, F.
\newblock Locally compact transformation groups and {C*}-algebras.
\newblock {\em Bull. Amer. Math. Soc.}, v. 73, n. 2, \newblock p. 222--226, 03
  1967.

\bibitem{Sauvageot}
SAUVAGEOT, J.-L.
\newblock Id{\'e}aux primitifs induits dans les produits crois{\'e}s.
\newblock {\em Journal of Functional Analysis}, v. 32, n. 3, \newblock p.
  381--392, 1979.

\bibitem{GootmanRosenberg}
GOOTMAN, E.; ROSENBERG, J.
\newblock The structure of crossed product {C*}-algebras: A proof of the
  generalized {E}ffros-{H}ahn conjecture.
\newblock {\em Inventiones mathematicae}, v. 52, n. 3, \newblock p. 283--298,
  1979.

\bibitem{Renault1987}
RENAULT, J.; SKANDALIS, G.
\newblock The ideal structure of groupoid crossed product {C*}-algebras.
\newblock {\em Journal of Operator Theory}, v. 25, n. 1, \newblock p. 3--36,
  1991.

\bibitem{IonescuWilliams}
IONESCU, M.; WILLIAMS, D.
\newblock The generalized {E}ffros-{H}ahn conjecture for groupoids.
\newblock {\em Indiana University Mathematics Journal}, v. 58, n. 6, \newblock
  p. 2489--2508, 2009.

\bibitem{DokuchaevExel}
DOKUCHAEV, M.; EXEL, R.
\newblock The ideal structure of algebraic partial crossed products.
\newblock {\em London Math. Soc.}, v. 115, n. 1, \newblock p. 91--134, 2017.

\bibitem{Steinberg}
STEINBERG, B.
\newblock A groupoid approach to discrete inverse semigroup algebras.
\newblock {\em ArXiv e-prints}, 2009.

\bibitem{Steinberg2}
STEINBERG, B.
\newblock Simplicity, primitivity and semiprimitivity of {\'e}tale groupoid
  algebras with applications to inverse semigroup algebras.
\newblock {\em J. Pure Appl. Algebra}, v. 220, \newblock p. 1035--1054, 2016.

\bibitem{Beuter}
BEUTER, V.; GON{\c{C}}ALVES, D.
\newblock The interplay between {S}teinberg algebras and partial skew rings.
\newblock {\em Journal of Algebra}, v. 497, \newblock p. 337--362, 05 2017.

\bibitem{Hazrat2018}
HAZRAT, R.; LI, H.
\newblock Graded {S}teinberg algebras and partial actions.
\newblock {\em Journal of Pure and Applied Algebra}, v. 222, \newblock p.
  3946--3967, 2018.

\bibitem{Exel2011}
EXEL, R.
\newblock Noncommutative cartan subalgebras of {C*}-algebras.
\newblock {\em The New York Journal of Mathematics [electronic only]}, v. 17,
  \newblock p. 331--382, 2011.

\bibitem{patersonbook}
PATERSON, A.
\newblock {\em {G}roupoids, {I}nverse {S}emigroups, and their {O}perator
  {A}lgebras}.
\newblock Birkh\"{a}user Basel, 1999.
\newblock v.  170 of {\em Progress in Mathematics}.

\bibitem{ExelCombinatorial}
EXEL, R.
\newblock Inverse semigroups and combinatorial {C*}-algebras.
\newblock {\em Bull. Braz. Math. Soc.}, v. 39, n. 2, \newblock p. 191--313,
  2008.

\bibitem{Lisa}
CLARK, L.; FARTHING, C.; SIMS, A.; TOMFORDE, M.
\newblock A groupoid generalization of {L}eavitt path algebras.
\newblock {\em Semigroup Forum}, v. 89, n. 3, \newblock p. 501--517, 2011.

\bibitem{ExelAmple}
EXEL, R.
\newblock Reconstructing a totally disconnected groupoid from its ample
  semigroup.
\newblock {\em Proceedings of the American Mathematical Society}, v. 138,
  \newblock p. 2991--3001, 2010.

\bibitem{ExelPardo}
EXEL, R.; PARDO, E.
\newblock The tight groupoid of an inverse semigroup.
\newblock {\em Semigroup Forum}, v. 92, n. 1, \newblock p. 274--303, 2016.

\bibitem{CordeiroBeuter}
BEUTER, V.; CORDEIRO, L.
\newblock The dynamics of partial inverse semigroup actions.
\newblock {\em ArXiv e-prints}, 2018.

\end{thebibliography}
